\documentclass[lewno11pt,smallextended,envcountsect,envcountsame,a4wide]{svjour3} 

\usepackage{graphicx}
\usepackage{mathrsfs} 
\usepackage{fix-cm} 
\usepackage{amsmath} 
\usepackage{amssymb}
\usepackage{latexsym}
\usepackage{dsfont} 
\usepackage{xcolor}
\usepackage{cite}
\usepackage{dsfont}
\usepackage{enumitem} 
\usepackage{caption}
\usepackage{subcaption}
\usepackage{float}

\newcommand{\Rx}{ {\mathbb{R}}_\infty}
\renewcommand{\Re}{\mathbb{R}}
\newcommand{\R}{\mathbb{R}}
\newcommand{\C}{\mathcal{C}}
\newcommand{\Y}{\mathcal{Y}}
\newcommand{\K}{\mathcal{K}}
\renewcommand{\P}{\mathbb{P}}
\renewcommand{\S}{\mathcal{S}}
\newcommand{\D}{\mathcal{D}}
\renewcommand{\epsilon}{\varepsilon}
\newcommand{\MoreauYosida}[2]{{\mathtt{e}}_{#2} #1}
\newcommand{\PMoreau}[3]{{\operatorname{\mathtt{e}}_{ {#2} }{ {#1}^{#3} } }}

\newcommand{\interior}{\,\operatorname{int}} 
\newcommand{\Prox}[2]{{\mathtt{Prox}}_{#2 #1}}
\newcommand{\dom}{\operatorname{dom}}
\newcommand{\G}[1]{G_{#1}}

\newcommand{\pmord}{\partial^{\tt b}}
\newcommand{\pfrech}{\partial^{\tt r}}

\newcommand{\norm}[1]{\left \Vert #1 \right \Vert}
\DeclareMathOperator{\Tr}{tr}
\newcommand{\media}{\mathtt{m}}
\renewcommand{\smartqed}{\flushright\qed}
\renewcommand{\Re}{\mathbb{R}}
\renewcommand{\norm}[1]{\left \Vert #1 \right \Vert}

\begin{document}

\title{Inner Moreau envelope of nonsmooth conic chance constrained optimization problems}
\titlerunning{Moreau envelope of nonsmooth chance constrained optimization problems}

\author{Wim van Ackooij \and Pedro P\'{e}rez-Aros \and Claudia Soto \and Emilio Vilches}

\institute{ W. van Ackooij \at
		EDF R\&D. OSIRIS
		7, Boulevard Gaspard Monge, F-91120 Palaiseau, France\\
		\email{wim.van-ackooij@edf.fr} 
		\and
		P. P\'erez-Aros \at
       Instituto de Ciencias de la Ingenier\'ia, Universidad de O'Higgins \\
        \email{pedro.perez@uoh.cl}      
     \and
      C. Soto \at
      Departamento de Ingenier\'ia Matem\'atica, 
      Universidad de Chile\\
       \email{csoto@dim.uchile.cl}
     \and 
      E.Vilches \at
      Instituto de Ciencias de la Ingenier\'ia,
      Universidad de O'Higgins \\
       \email{emilio.vilches@uoh.cl}
}
\maketitle

\begin{abstract}

Optimization problems with uncertainty in the constraints occur in many applications. Particularly, probability functions present a natural form to deal with this situation. Nevertheless, in some cases, the resulting probability functions are nonsmooth. This motivates us to propose a regularization employing the Moreau envelope of a scalar representation of the vector inequality. More precisely, we consider a probability function which covers most of the general classes of probabilistic constraints: $$\varphi(x)=\P(\Phi(x,\xi)\in -\K),$$ where $\K$ is a convex cone of a Banach space. The conic inclusion $\Phi (x,\xi) \in - \mathcal{K}$  represents an abstract system of  inequalities,  and $\xi$ is a random vector . We propose a regularization by applying the Moreau envelope to the scalarization of the function $\Phi$. In this paper, we demonstrate, under mild assumptions, the smoothness of such a regularization and that it satisfies a type of variational convergence to the original probability function. Consequently, when considering an appropriately structured problem involving   probabilistic constraints, we can thus entail the convergence of the minimizers of the regularized approximate problems to the minimizers of the original problem.  Finally, we illustrate our results with examples and applications in the field of (nonsmooth) joint, semidefinite  and probust chance constrained optimization problems.
 
  
\end{abstract}	

\section{Introduction}

Chance constrained optimization arises as an essential topic in stochastic optimization because it presents an intuitive tool to deal with uncertainty in objective functions or constraints being a middle point between expectation and robustness in the model. Nevertheless, the resulting class of problems might be nonsmooth and/or nonconvex. In any case, they are frequently stated to be very challenging from a practical point of view. 

\noindent This has led to many different approaches for handling probability functions. First, investigations regarding the convexity of upper level sets to probability functions have received quite some attention, starting with the classic works on log-concavity by Prékopa, e.g., \cite{Prekopa_1971} and leading up to recent contributions regarding ``eventual convexity", e.g., \cite{Henrion_Strugarek_2008,vanAckooij_Laguel_Malick_Matiussi-Ramalho_2022}. The topic of understanding differentiability of probability functions (or the kind of differentiability) has also received great attention. Here we can indicate, e.g., \cite{Marti_1995,Uryasev_1995} but also the recent contributions, e.g., \cite{Royset_Polak_2004,vanAckooij_Henrion_2014,Hantoute_Henrion_Perez-Aros_2017}. For a recent introductory text to the topic, we refer to \cite{vanAckooij_2020}. Dedicated algorithms for the handling of optimization problems under probability constraints have also received much attention. Generally they can be subdivided according to the nature of the random vector itself: continuous or discrete. Nonetheless, a popular strategy consists in replacing the continuous random vector with a discrete sample, thus going from the first to the second class. The second subdivision resides in the choice to either handle the probability function as itself, while the second consists of replacing it with an appropriate substitute. The first strategy requires explicit understanding of the nature of the underlying problem, e.g., convexity, differentiability and generally builds on non-linear programming approaches, e.g., \cite{Szantai_1988,Mayer_2000,Bremer_Henrion_Moller_2015,vanAckooij_Sagastizabal_2014,vanAckooij_Oliveira_2016}, all while exploiting the available knowledge. A popular strategy for the ``substitution'' path is replacing the probability function by a different mapping. The underlying idea is to write the probability function as an expectation and then to replace the indicator function by a parametrized approximation. These and other related ideas can be found in e.g., 
\cite{Hong_Yang_Zhang_2011,Nemirovski_Shapiro_2006a,Geletu_Hoffmann_Kloppel_Li_2017,Shan_Zhang_Xiao_2014,Pena-Ordieres_Luedtke_Wachter_2020}. One of the key difficulties in such an approach is to ensure proper convergence of the approximated solutions (value, solution vectors) to ``candidates" of the original problem formulation. Evidently when the ``parameters" reach their limiting value, the numerical properties of the approximations become ``degenerate" and hence a trade-off between ``precision'' and ``computability" has to be found. Likewise when the probability function acts on several inequalities (the so-called ``joint case"), this too can lead to a second step of approximations, e.g., employing a smooth approximation of the maximum function. The current work suggests something midway: a series of approximations based on the use of reliable and well developed Moreau envelope.

\noindent In convex analysis, the Moreau envelope (also called Moreau-Yosida regularization) is a a useful regularization for general nonsmooth convex functions. The applications of such an envelope covers a variety of theoretical developments, and it is at the core of many numerical optimization methods. Nowadays, there are plenty of explicit formulations for the computation of the Moreau envelope of most common convex functions, and there are efficient algorithms to compute the envelope numerically for more complex data (see, e.g., \cite{Combettes2017,MR3719240} and the references therein). 

\noindent This paper aims to propose and investigate a general regularization of probabilistic functions, which employs the Moreau envelope of some functions. Formally, in this work, we  consider a probability function $\varphi: \mathcal{H}\to [0,1]$ given by
\begin{equation}\label{Proba:funct}
	\varphi(x) := \P\left( \omega \in \Omega : \Phi(x,\xi(\omega)) \in -\K  \right),
\end{equation}
where $\mathcal{H}$ is a Hilbert space, $\xi : \Omega \to \R^m$ is an $m$-dimensional random vector, $\K\subset \Y$ is a (nonempty) convex cone of a Banach space $\Y$ and  $\Phi :\mathcal{H}\times  \mathbb{R}^m \to \Y$ is a function. Here, it is worth mentioning that the formulation of the probability function $\varphi$ in \eqref{Proba:funct} covers several of the most general classes of probabilistic constraints arising in chance constrained, joint-chance constrained, and even probabistic/robust (probust) chance constrained optimization problems, as we will show in this work. Here, it is important to mention that the inclusion can be represented as an abstract inequality given by the cone order $x \preceq y$ if and only if $ y -x\in \mathcal{K}$. A particular example covered is one wherein $\K$ is the cone of positive definite matrices, and thus $\Phi(x,\xi)\in \mathcal{K}$ represents that our (random) decision matrix $ \Phi(x,\xi)$  should be positive semidefinite for most possible cases (see Section \ref{sec:examples} for more details on such an application).

 Since the random (possibly infinite dimensional) constraint $\Phi(x,\xi(\omega)) \in -\K$ is challenging to handle, we propose a Moreau regularization of a (nonsmooth) scalarization of the function $\Phi$. Then our regularization will be given by the probability function generated by the Moreau envelope of that regularization (see Section \ref{innerScalarization} for more details).  
 Surprisingly, and under mild assumptions, such regularization inherits variational properties of the Moreau envelope, for instance, its smoothness and variational convergence to the original function. Those properties are used to provide a regularization of (general) chance constrained optimization problems and the convergence of the minimizers of the regularized problems to the minimizers of the original formulation. It is natural to understand such convergence as a naive form to propose a toolbox for solving general classes of nonsmooth chance constraints optimization problems. Consequently, our developments open a gate to study further improvements using the ideas exploited in deterministic optimization algorithms, which use Moreau envelops of functions in a future research project.  

This paper is organised as follows: section \ref{sec:preliminary} provides background information regarding notation, frequently used results and suggests the setting of the work. Section \ref{sec:varcvg} examines the convergence of the inner Moreau envelope of the probability function towards the nominal probability function \eqref{Proba:funct}. Differentiability of the approximating function is investigated in section \ref{sec:diff}. The manner in which the use of approximated probability functions, through their inner Moreau envelope, allow us to approximate a given optimization problem is investigated in section \ref{sec:consistency}. Section \ref{sec:examples}  provides several examples and possible applications of the developed results. Finally, the paper ends with some conclusions and perspectives for future research projects.  

\section{Preliminaries}
\label{sec:preliminary}

In this section, after having introduced notation and base concepts used throughout the manuscript, we introduce formally the Moreau envelope  of given mappings. We also provide elementary results and properties frequently used of this envelope. The subsequent section introduces details about the random vectors themselves, as well as the possibility of representing a probability function, through a spherical-radial decomposition. The latter allows us, under certain structural assumptions, to state known results regarding differentiability of certain probability functions. Finally, this section is dedicated to the introduction of the inner Moreau envelope of probability functions of the form \eqref{Proba:funct}. The resulting object will be investigated in the remainder of the paper.

\subsection{Notation}
Let $(\mathcal{H},\langle \cdot,\cdot\rangle)$ be a separable Hilbert space with unit ball $\mathbb{B}$ and $\Y$ will be a Banach space. Given a set $C$ and a topology $\tau$, we denote by $\operatorname{cl}^{\tau} C$, $\interior_{\tau}{C}$, the closure, the interior of $C$ with respect to $\tau$. When there is no confusion, we omit the symbol $\tau$. The cone generated by $C$ is denoted by $\operatorname{cone}C$. For a given cone $\mathcal{K}\subset \Y$, we denote its positive and negative dual cone by
\begin{align*}
   \mathcal{K}^{+}&:=\{ v^\ast \in \Y^\ast: \langle v^\ast, v\rangle \geq 0 \text{ for all }v\in \K \},\\
   \mathcal{K}^{-}&:=\{ v^\ast \in \Y^\ast: \langle v^\ast, v\rangle \leq 0 \text{ for all }v\in \K \},
\end{align*}
respectively.
We set $\mathbb{B}_r(x)$ the ball with center at $x$ and with radius $r$. The space $\mathcal{H}\times \mathbb{R}^m$ is considered as a Hilbert space with the appropriate inner product. 

\noindent The indicator function of a set $C\subset \mathcal{H}$ is the function $\delta_C\colon \mathcal{H}\to \mathbb{R}\cup\{+\infty\}$ defined by $\delta_C(x)=0$ for $x\in C$ and $\delta_C(x)=+\infty$ otherwise.

\noindent The epigraph and the (effective) domain of an extended real valued function $\psi:\mathcal{H}\to\Rx$, where $\Rx:=\mathbb{R}\cup\{+\infty\}$,  are defined and denoted respectively by $\operatorname{epi}\psi:=\{(x,\lambda)\in \mathcal{H}\times \mathbb{R}\colon \psi(x)\leq \lambda\}$ and
$\dom \psi:=\{x\in\mathcal{H}\colon \psi(x)<+\infty\}$. 

\noindent The convex subdifferential of an extended real valued function $\psi:\mathcal{H}\to\Rx$ is defined and denoted by 
$\partial \psi(x)=\{z\in \mathcal{H}: \psi(y)\geq \psi(x)+\langle z, y-x\rangle \textrm{ for all } y\in\mathcal{H}\}$ when $x\in \operatorname{dom}\psi$. We set  $\partial \psi(x)=\emptyset$ when $\psi(x)=+\infty$. The Legendre-Fenchel conjugate of $\psi$ is the extended real valued function $\psi^*:\mathcal{H}\to\Rx$ defined by $\psi^*(z)=\sup_{y\in \mathcal{H}}\{\langle z,y\rangle-\psi(y)\}.$ The set of all convex, proper, and lower-semicontinuous (lsc) functions is denoted by $\Gamma_0(\mathcal{H})$.

\noindent A function $\psi\colon \mathcal{H}\to \Rx$ is coercive if the sets $\{x\in\mathcal{H}\colon \psi(x)\leq \alpha\}$ are bounded for all $\alpha\in\R$. Moreover, for $\psi \in \Gamma_0(\mathcal{H})$, the above is equivalent to the condition $0\in\text{int}(\dom \psi^*)$ (see, e.g., \cite[Proposition 14.16]{Combettes2017}). We say that a function $\psi$ has a strong minimum at $x_0$, if   $\psi(x) > \psi(x_0)$ for all $x \neq x_0$, and $x_k \to x_0$ whenever $\psi(x_k) \to \psi(x_0)$.

\noindent A sequence of sets $(C_k)_k\subset \mathcal{H}$ \emph{Painlev\'e-Kuratowski} converges to a set $C$ if the following conditions hold:
\begin{enumerate}
     \item[$a)$] $ C\subset \liminf\limits_{ k\to\infty}C_k :=\left\{ x\in \mathcal{H}\colon \exists x_k \in C_k \text{ with } x_k \to x  \right\}$, and
 		\item[$b)$] $\limsup\limits_{k\to\infty}C_k:=\left\{ x \in \mathcal{H}\colon \exists x_{n} \in C_{k_n} \text{ with } k_n \to \infty \text{ and } x_{n} \rightarrow x \right\} \subset C$.
 	\end{enumerate}
The sequence $(C_k)_k$ is said to \emph{Mosco converge} to $C$ if condition $a)$ is satisfied and $b)$ is replaced by the following condition:
\begin{enumerate}
 		\item[$c)$] $w$-$\limsup\limits_{k\to\infty}C_k:=\left\{ x \in \mathcal{H}: \exists x_{n} \in C_{k_n} \text{ with } k_n \to \infty \text{ and } x_{n} \rightharpoonup x \right\} \subset C$.
\end{enumerate}
Moreover, the limit set of a sequence of epigraphs is again an epigraph (in both of the above notions). Thus, we obtain two notions of convergence of functions which can be characterized as follows: A sequence of functions $\psi_k:\mathcal{H}\to\Rx$ \emph{epi-converge} to $\psi:\mathcal{H}\to \Rx$ when the following two conditions hold:
\begin{enumerate}
   \item[$a')$] For all $x\in \mathcal{H}$, there exist $x_k\to x$ such that $\limsup\limits_{ k \to \infty }  \psi_k(x_k)\leq \psi(x)$, and
   \item[$b')$] For all $x\in \mathcal{H}$ and for all $x_k\to x$, we have $\liminf\limits_{ k \to \infty }  \psi_k(x_k)\geq \psi(x)$.
\end{enumerate} 
The sequence $\psi_k$ is said to \emph{Mosco epi-converge} to $\psi$ when condition $a')$ is satisfied and $b')$ is replaced by the following condition:
\begin{enumerate}
   \item[$c')$] For all $x\in \mathcal{H}$ and for all $x_k\rightharpoonup x$, we have $\liminf\limits_{ k \to \infty }  \psi_k(x_k)\geq \psi(x)$.
\end{enumerate} 
Hypo-convergences notions can be obtained by applying the above notions to the functions $-\psi,-\psi_k$. Moreover, a sequence of functions $(\psi_k)$ \emph{converges continuously} to $\psi$ if $(\psi_k)_k$ epi-converges and hypo-converges to $\psi$, i.e., for all $x_k\rightarrow x$ we have that $\lim_{k\rightarrow \infty}\psi_k(x_k)=\psi(x)$.

\subsection{Moreau envelope}
Given a function $\psi \in  \Gamma_0(\mathcal{H})$ and $\lambda >0$, the Moreau envelope of $\psi$ of parameter $\lambda $ is the function $\MoreauYosida{\psi}{\lambda}: \mathcal{H}\to \Rx$ defined by
\begin{equation*}
   \MoreauYosida{\psi}{\lambda}(x) := \inf_{ z\in \mathcal{H}} \left( \psi(z) + \frac{1}{2\lambda}  \| x-z  \|^2 \right).
\end{equation*}
The above infimum is attained at a unique point, which is called the  proximal point of $\psi$ of index $\lambda$ at $x$. It defines a nonexpansive operator $ \Prox{\psi}{\lambda } :\mathcal{H}\to\mathcal{H}$ given by 
\begin{align*}
\Prox{\psi}{\lambda }(x) &:=\operatorname{argmin}_{ z\in  \mathcal{H}}  \left( \psi(z) + \frac{1}{2\lambda}  \| x-z  \|^2 \right) 
= \left( I + \lambda \partial \psi \right)^{-1}(x)
\end{align*} 
Moreover, the Moreau envelope of any proper lower-semicontinuous function is convex and continuously differentiable with 
\begin{equation}
\label{eq:derivativeofMoreau}
\nabla \MoreauYosida{\psi}{\lambda}(x)=\frac{1}{\lambda}(x-\Prox{\psi}{\lambda}(x)) \textrm{ for all } x\in \mathcal{H}.
\end{equation}
It follows moreover from the above identification of the proximal operator with a resolvant that (see, e.g., \cite[Proposition 16.44]{Combettes2017}):
\begin{equation}
\label{eq:nablaprox}
\nabla \MoreauYosida{\psi}{\lambda}(x) \in \partial \psi( \Prox{\psi}{\lambda }(x) ).
\end{equation}

The following proposition summarizes some properties of the Moreau envelope in Hilbert spaces. We refer to \cite{Attouch2014,Combettes2017,MR123071} for more details.

\begin{proposition}\label{basicprop} Let $g\colon \mathcal{H}\to \R$ be a convex and lower semicontinuous function. Then the following hold.
	\begin{enumerate}[label=\alph*)]
		\item Monotone convergence: $  \MoreauYosida{g}{\lambda} (x) \nearrow g(x) $ as $\lambda \searrow \ 0$ for all $x\in \mathcal{H}$. 
		\item Convergence of resolvents: $\Prox{g}{\lambda}(x)\to x$ as $\lambda\to 0$ for all $x\in \mathcal{H}$.
		\item Lower epi-convergence: If $x_k \rightharpoonup x $ and $\lambda_k \searrow 0$, then 
		\begin{equation*}
		  g(x) \leq \liminf\limits_{ k \to \infty }  \MoreauYosida{g}{\lambda_k} (x_k).
		\end{equation*}
			\item Continuous convergence: If $x_k \rightarrow x $ and $\lambda_k \searrow 0$, then \begin{equation*}
			  g(x) =\lim_{ k \to \infty }  \MoreauYosida{g}{\lambda_k} (x_k).
			\end{equation*}
	\end{enumerate}
\end{proposition}

\begin{proof}
Items a-c) can be found in \cite[Proposition~2.2]{Perez-Vilches2021}. Let us focus on $d)$. To this end, let a sequence $x_k\to x$ and $\lambda_k \searrow 0$ be given. Then, by virtue of the proximal operator being non-expansive we have
\begin{equation}\label{prox-n}
\begin{aligned}
\Vert \Prox{g}{\lambda_k}(x_k)-x\Vert &\leq \Vert \Prox{g}{\lambda_k}(x_k)-\Prox{g}{\lambda_k}(x)\Vert +\Vert \Prox{g}{\lambda_k}(x)-x\Vert\\
& \leq \Vert x_k-x\Vert +\Vert \Prox{g}{\lambda_k}(x)-x\Vert,
 \end{aligned}
\end{equation}
which, by b), implies that $\Prox{g}{\lambda_k}(x_k)\to x$, as $k\to +\infty$. Moreover, for all $k\in \mathbb{N}$
\begin{align*}
g(\Prox{g}{\lambda_k}(x_k))&\leq g(\Prox{g}{\lambda_k}(x_k))+\frac{1}{2\lambda_k}\Vert x_k-\Prox{g}{\lambda_k}(x_k)\Vert^2\\
&= \MoreauYosida{g}{\lambda_k} (x_k)\leq g(x_k),
\end{align*}
where a) was used to derive the last inequality.
Thus, by using \eqref{prox-n}, the continuity of $g$ ($g$ is lower semicontinuous with finite values) and taking the limit $k\to +\infty$ in the latter inequality, we obtain that
$\lim_{k\to +\infty}\MoreauYosida{g}{\lambda_k} (x_k)=g(x).
$\smartqed
\end{proof}

The final proposition in this section gives a precise (uniform) bound on the distance between a function and its Moreau envelope in finite dimensional setting.

\begin{proposition}\label{Uniform:aprox}
Let $S$ be a closed, convex and bounded subset of $\mathcal{H}=\mathbb{R}^n$ and let $g\colon \mathcal{H}\to \R$ be a convex  function. Then, there exist $\ell\geq 0$, $\lambda_0\in (0,1)$ and a constant $C>0$ such that for all $\lambda\in (0,\lambda_0)$, the function $x\mapsto \MoreauYosida{g}{\lambda}(x)$ is $\ell$-Lipschitz on $S$ and 
\begin{equation*}
	\sup_{x\in S} | \MoreauYosida{g}{\lambda}(x) -g(x)| \leq  \ell \sqrt{\lambda}C.
\end{equation*}
\end{proposition}
\begin{proof} Let us consider the set $\tilde{S}:=\{x\in \mathcal{H}\colon d_S(x)\leq 1\}$. Since $\tilde{S}$ is closed, convex and bounded, there exists $\ell\geq 0$ such that $g$ is $\ell$-Lipschitz on $\tilde{S}$. Moreover, for all $x\in S$
\begin{equation*}
\begin{aligned}
\frac{1}{2\lambda}\Vert x-\Prox{g}{\lambda}(x)\Vert^2
\leq& g(x)-g(\Prox{g}{\lambda}(x))\\
\leq& g(x)-\langle x^{\ast},\Prox{g}{\lambda}(x)\rangle -\beta\\
\leq& g(x)+\Vert x^{\ast}\Vert \cdot \Vert x-\Prox{g}{\lambda}(x)\Vert+\Vert x^{\ast}\Vert\cdot \Vert x\Vert-\beta\\
\leq& g(x)+\lambda \Vert x^{\ast}\Vert^2+\frac{1}{4\lambda}\Vert x-\Prox{g}{\lambda}(x)\Vert^2\\
&+\Vert x^{\ast}\Vert\cdot\Vert x\Vert-\beta,
\end{aligned}
\end{equation*}
where $x\mapsto \langle x^{\ast},x\rangle +\beta$ is an arbitrary but fixed affine minorant of $g$. We have also used the inequality $ab \leq \frac{c^2}{2} a^2 + \frac{1}{2c^2}b^2$ for $c=\sqrt{2\lambda}$. Thus, for all $x\in S$, we have
\begin{equation*}
\begin{aligned}
\Vert x-\Prox{g}{\lambda}(x)\Vert^2&\leq 4\lambda\left( g(x)+\lambda \Vert x^{\ast}\Vert^2+\Vert x^{\ast}\Vert\cdot\Vert x\Vert-\beta\right).
\end{aligned}
\end{equation*}
Since the right-hand side of the latter inequality is uniformly bounded in $S$, $\lambda \leq 1$, it is possible to find a constant $C>0$ such that
\begin{equation}\label{eq-C}
\Vert x-\Prox{g}{\lambda}(x)\Vert \leq \sqrt{\lambda} C.
\end{equation}

In particular, it is possible to find $\lambda_0\in (0,1)$ such that for all $\lambda \in (0,\lambda_0)$
$$
\Vert x-\Prox{g}{\lambda}(x)\Vert\leq 1 \textrm{ for all } x\in S,
$$
implying that $\Prox{g}{\lambda}(x)\in \tilde{S}$. Hence, for all $\lambda\in (0,\lambda_0)$ and $x\in S$
\begin{align*}
0\leq g(x)-\MoreauYosida{g}{\lambda}(x)&=g(x)-g(\Prox{g}{\lambda}(x))-\frac{1}{2\lambda}\Vert x-\Prox{g}{\lambda}(x)\Vert^2\\
&\leq \ell \Vert x-\Prox{g}{\lambda}(x)\Vert \leq \ell \sqrt{\lambda}C,
\end{align*}
where we have used \eqref{eq-C} and the fact that $\Prox{g}{\lambda}(x)\in \tilde{S}$ for all $x\in S$ when $\lambda \in (0,\lambda_0)$. Finally, since $\MoreauYosida{g}{\lambda}$ is convex and differentiable, for all $x\in S$
$$
\MoreauYosida{g}{\lambda}(y)\geq \MoreauYosida{g}{\lambda}(x)+\langle \nabla \MoreauYosida{g}{\lambda}(x),y-x\rangle \textrm{ for all } y\in \mathcal{H},
$$
where $\nabla \MoreauYosida{g}{\lambda}(x)\in \partial g(\Prox{g}{\lambda}(x))$. Hence, since $\Prox{g}{\lambda}(x)\in \tilde{S}$ and $g$ is $\ell$-Lipschitz on $\tilde{S}$, it follows that $\Vert \nabla \MoreauYosida{g}{\lambda}(x)\Vert \leq \ell$ for all $x\in S$. Therefore, for all $x,y\in S$
$$
\MoreauYosida{g}{\lambda}(x)\leq \MoreauYosida{g}{\lambda}(y)+\ell \Vert y-x\Vert, 
$$
which ends the proof, by showing that $x\mapsto \MoreauYosida{g}{\lambda}(x)$ is $\ell$-Lipschitz on $S$.
\smartqed
\end{proof}

\subsection{Spherical radial decomposition and gradient formula for probability functions} \label{subsec:spherical_descomposition_gradient_formula}

Let $(\Omega,\mathcal{F},\mathbb{P})$ be a probability space. In what follows, $\xi : \Omega  \to \mathbb{R}^m$ is a $m$-dimensional random vector admitting a (continuous) density with respect to the Lebesgue measure, which is denoted by $f_{\xi}$. 

 Consider a continuously differentiable function $g :  \mathcal{H}
	\times \R^m \to \R$, which is convex with respect to its second variable. Let us consider the probability function
\begin{equation}\label{probfunc:g}
	\varphi(x):=\mathbb{P}( \omega \in \Omega : g(x,\xi(\omega))
	\leq 0 ).
\end{equation}
The spherical-radial decomposition \cite{Hantoute_Henrion_Perez-Aros_2017,vanAckooij_Malick_2017,Fang_Kotz_Ng_1990} allows to rewrite the above probability as
\begin{equation*}
	\varphi(x)= \int\limits_{v \in \mathbb{S}^{m-1}} e(x,v) d
	\mu_\zeta(v),
\end{equation*}
where $e: \mathcal{H}\times \mathbb{S}^{m-1} \to \Rx$ is the \emph{radial probability-like} function given
by
\begin{align*}
	e(x,v) =  \frac{2 \pi^{\frac{m}{2} } |\det(L)|}{ \Gamma(
		\frac{m}{2} )   } \displaystyle  \int\limits_{  \{ r\geq 0 :
		g(x, rLv)   \leq 0   \}   }  r^{m-1}f_\xi (rLv)
	dr.
\end{align*}
Whenever $x \in \mathcal{H}$ is such that $\varphi(x) \geq \frac 12$, e.g., \cite[Corollary 2.1]{vanAckooij_Malick_2017}, and $\xi$ attributes at least half probability to any half space containing $\media := \mathbb{E}(\xi)$, then it must follow that $g(x,\mu) < 0$ when $g$ itself admits a Slater-point or is qualified, i.e.,
$$\interior\{z \in \Re^m \; : g(x,z) \leq 0\} = \{z \in \Re^m \; : g(x,z) < 0\}.$$ 
The last condition is in any case needed in order to ensure continuity of the probability function, e.g., \cite[eq. (3)]{Farshbaf-Shaker_Henrion_Homberg_2017}. In the sequel we will assume that $\media := \mathbb{E}(\xi) = 0$ and that moreover it is true that $g(x,\media) < 0$. Then convexity in the second argument of $g$, ensures (see, e.g., \cite{vanAckooij_Henrion_2014}) that 
\begin{align*}
\{r \geq 0 \; : g(x, rLv) \leq 0 \} = [0,\rho(x,v)],\;  \text{where } \rho(x,v) :=\sup\{ r :g(x, rLv) \leq 0\},  
\end{align*}  
where $\rho(x,v) = \infty$ is allowed,  with the convention $[0,+\infty]= [0,+\infty)$.

\noindent In order to simplify the notation, let us define the
\emph{density-like function} $\theta$:
\begin{align}
\label{transformationftorho}
	\theta(r,v):=  \frac{2\pi^{\frac{m}{2}} |\textnormal{det}(L)|
	}{\Gamma(\frac{m}{2})} r^{m-1}f_\xi (rLv).
\end{align}
We will associate the
finite and infinite directions with respect to $g$ as
the sets defined by
\begin{align*}
	F(x) &:=\{v\in \mathbb{S}^{m-1}\ :\ \exists r\geq 0:g(x,rLv)=0\}, \text{ and } 	I(x) :=\mathbb{S}^{m-1} \backslash 	F(x)
\end{align*}%
respectively. We can observe that $F(x) = \dom(\rho(x,.))$, with $\rho(x,.)$ as introduced earlier.

The following technical condition is used to obtain formulae for the gradient of probability functions (see \cite{vanAckooij_Perez-Aros_2022}). It is worth to emphasize that the condition below is general enough to cover most of the known distributions. For example, for Gaussian distributions it holds under an exponential growth condition on the gradients $\nabla_x g$. 

\begin{definition}[$\eta_{\theta}$-growth condition]
Consider $ \bar{x} \in \mathcal{H}$ and $\bar{v} \in
	I( \bar{x})$. And let a mapping $\eta_{\theta} : \R\times \mathbb{S}^{m-1}
	\to [0,+\infty]$ be such that
	\begin{equation}\label{thetagrowth}
		\lim\limits_{\substack{r \to +\infty \\
				v \to \bar{v} } } r {\theta} (r,v) \eta_{\theta}(r,v) =0.
	\end{equation}
	We say that the mapping $ g$ satisfies
	the $\eta_{\theta}$-growth condition at $(\bar{x}, \bar{v})$ if
	for some $\epsilon, l >0$
	\begin{equation*}
		\| \nabla_x g(x,rLv) \| \leq l \eta_{\theta}(r,v),  \;
		\forall (x,v) \in \mathbb{B}_{\epsilon} (\bar{x}) \times
		\mathbb{B}_{\epsilon} (\bar{v}), \;\; \forall r \geq l. 
	\end{equation*}
\end{definition}

 The next result  corresponds to a gradient formula for the probability function using the spherical radial decomposition. We refer to \cite{vanAckooij_Perez-Aros_2022}, for similar results and further extensions of the next theorem.
\begin{theorem}[Corollary 3.2 \cite{vanAckooij_Perez-Aros_2022}]
	\label{thm:clarkeappli} Let $\mathcal{H}$ be a finite-dimensional Hilbert space, let $\bar{x} \in \mathcal{H}$ be such that $	g(\bar x,0) < 0$, and assume that $g$ 
	satisfies the $\eta_{\theta}$-growth
	condition at $(\bar{x},\bar{v})$ for all $\bar{v} \in I(\bar{x})$.
	Then the probability function $\varphi$ defined in \eqref{probfunc:g} is continuously differentiable on an appropriate neighbourhood $U'$ of
	$\bar x$ with
	\begin{align*}
		\nabla \varphi (x)= \int\limits_{  \mathbb{S}^{m-1}} 	\nabla_x e(x,v)  d\mu_\zeta(v) \textrm{ for all } x \in U',
	\end{align*}
	where, 
	\begin{align*}
		\nabla_x e(x,v) &=\left\{  \begin{array}{cc}
		 -    \frac{ \theta(\rho(x,v),v) 
		 }{%
		 	\left\langle  \nabla_z g(x,\rho(x,v)Lv),Lv\right\rangle } \nabla_x g(x,\rho(x,v)Lv)  &  \text{ if }  v\in F(x),\\
	 	 & \\
	 	0 & \text{ if }  v\in I(x) 
		\end{array}   		
		\right.
	\end{align*}
\end{theorem}

\subsection{Inner scalarization of $\varphi$}
\label{innerScalarization}

In this subsection, we describe our inner regularization of the probability function \eqref{Proba:funct}. In order to set up a suitable framework to use the properties of the Moreau envelope we need to impose that our nominal function $\Phi$ in \eqref{Proba:funct} satisfies some convexity properties. A common assumption in the study of probability functions is that the inequality systems satisfies some property of convexity with respect to the random variable $\xi \in  \R^m$, but not necessarily in the decision variable $x\in \mathcal{H}$. Since, our function $\Phi$ is vector-valued, we will suppose that some scalarizations are convex up to the addition of a smooth convex function. Formally, let us consider a (weak$^\ast$-)compact convex set $\C \subseteq \Y^\ast$, which generates the positive polar cone of $\K$, that is, 
\begin{align}\label{generateK}
\operatorname{cl}^{w^\ast} \operatorname{cone} \C = \K^{+}
\end{align}
In what follows, we assume that there is a continuously differentiable convex function $h: \mathcal{H}\to \R$ such that for all $v^\ast \in \C$, the function 
\begin{align}\label{scalar01}
   \mathcal{H}\times \mathbb{R}^m \ni (x,z) \to 	\Phi_{v^\ast}^h (x,z) :=\langle v^\ast,\Phi \rangle (x,z) + h(x)
\end{align}
is convex in both variables, where $\langle v^\ast,\Phi \rangle (x,z) := \langle v^\ast,\Phi(x,z) \rangle$.

\begin{example}[Separated variables in joint chance constrained optimization]
Let us consider the probability function $\varphi(x)=\mathbb{P}(\omega\in \Omega \colon g(x,\xi(\omega))\leq 0)$, where 
$g: \R^n \times \R^m \to \R^s$ is the function defined by $g(x,\xi) = \Psi(x) + A\xi$, where $A$ is a matrix and $\Psi\colon \mathbb{R}^n \to \mathbb{R}^s$ a $C^2$ function. If we set 
$\Phi(x,z) = \Psi(x) +Az$,  $\mathcal{K}:=\R^s_+$, and $\mathcal{C}$ is any convex compact set with $\operatorname{cl} \operatorname{cone}\mathcal{C}=\R^s_{+}$, then $\Phi$ satisfies \eqref{scalar01}. Indeed, since $\Psi=(\Psi_1,\ldots,\Psi_s)$ is $C^2$, there are $C^2$ convex functions $\psi_1^k$ and $\psi_2^k$, for $k=1,\ldots,s$,  such that $\Psi_k=\psi_1^k-\psi_2^k$ (see, e.g.,  \cite{1982Pommelet,MR873269,Oliveira_2020}). Hence, since $\mathcal{C}$ is compact, there exists $C>0$ such that $\Phi$ satisfies \eqref{scalar01} with $h=C\sum_{k=1}^s \psi_2^k$.
\end{example}
Next, let us introduce the supremum function $ S^h_\Phi : \mathcal{H}\times  \mathbb{R}^m \to \R$ given by
\begin{equation}\label{supfunctionS}
	S^h_\Phi(x,z):=\sup\left\{ 	\langle v^\ast,\Phi \rangle (x,z) + h(x) : v^\ast \in \mathcal{C}        \right\}.
\end{equation}
Moreover, for $h = 0$, we simply write $S_\Phi:=S^0_\Phi$. 

The next proposition enables us to rewrite the probability function \eqref{Proba:funct} in terms of the supremum function \eqref{supfunctionS}.
\begin{proposition}
Let  $\mathcal{H}$ be a separable Hilbert space, $\xi : \Omega \to \R^m$ be a random vector, 
$\K\subset \Y$ be a (nonempty) convex cone of a (possibly nonseparable) Banach space and  $\Phi :\mathcal{H}\times  \mathbb{R}^m \to \Y$ be a function such that \eqref{scalar01} holds. Then, 
\begin{equation}\label{Proba:funct:2}
		\varphi(x) = \mathbb{P}\left( \omega \in \Omega : S^h_\Phi (x,\xi(\omega)) \leq h(x) \right) \text{ for all } x\in \mathcal{H}.
\end{equation}
\end{proposition}
\begin{proof}
Fix $ x\in \mathcal{H}$ and $ \omega \in \Omega$. Then, by the bipolar theorem (see, e.g., \cite[Theorem 3.38 p.99]{MR2766381})  we have that 
\begin{align*}
   \Phi(x,\xi(\omega) ) \in -\mathcal{K} & \Leftrightarrow  -\Phi(x,\xi(\omega) ) \in (\mathcal{K}^{-})^-  \\ 
   & \Leftrightarrow \langle v^\ast,-\Phi \rangle (x,\xi(\omega)) \leq 0, \forall v^\ast \in  \K^- \\
   & \Leftrightarrow \langle v^\ast,\Phi \rangle (x,\xi(\omega)) \leq 0, \forall v^\ast \in  \K^+ \\
   & \Leftrightarrow  	\langle v^\ast,\Phi \rangle (x,\xi(\omega)) \leq 0, \forall v^\ast \in \C\\
   & \Leftrightarrow  	\langle v^\ast,\Phi \rangle (x,\xi(\omega)) + h(x) \leq h(x),  \forall v^\ast \in  \C  \\
    & \Leftrightarrow  S^h_\Phi (x,\xi(\omega)) \leq h(x),
\end{align*}
where we used the fact that $\C$ generates the positive polar cone of $\K$ (see \eqref{generateK}), which proves \eqref{Proba:funct:2}.
\smartqed
\end{proof}

The previous formula \eqref{Proba:funct:2} for the probability function \eqref{Proba:funct} allows us to propose a inner regularization based on the Moreau envelope. Given $\lambda>0$, we define the inner regularization of $\varphi$ as
\begin{equation}\label{Proba:funct:lamb}
		\varphi_{\lambda}(x) := \mathbb{P}\left( \omega \in \Omega : \PMoreau{\Phi}{\lambda}{h} (x,\xi(\omega)) \leq h(x) \right),
\end{equation}
where $\PMoreau{\Phi}{\lambda}{h} := \MoreauYosida{ S^h_\Phi}{\lambda}$  is the Moreau envelope of the supremum function \eqref{supfunctionS}. It is worth to emphasize that the Moreau envelope of the supremum function \eqref{supfunctionS} is the supremum of Moreau envelopes of the scalarizations \eqref{scalar01}, which is established in the next result.
\begin{proposition}
Let $\Phi : \mathcal{H}\times \mathbb{R}^m \to  \Y$ be a continuous functions satisfying \eqref{scalar01} for some continuously differentiable convex function $h$. Then, for all $\lambda >0$
\begin{align*}
\PMoreau{\Phi}{\lambda}{h} (x,z) = \max\limits_{ v^\ast \in  \C } \MoreauYosida{\Phi_{v^\ast}^h}{\lambda}(x,z)  \text{ for all } (x,z) \in \mathcal{H}\times \mathbb{R}^m.
\end{align*}
\label{prop:moreaumaxismaxmoreau}
\end{proposition}
\begin{proof}
By virtue of \eqref{scalar01}, it is clear that the function $(x,z,v^\ast) \to \Phi_{v^\ast}^h(x,z)$ is, by assumption, convex with respect to $(x,z)$, and readily seen to be concave with respect to $v^\ast\in \mathcal{C}$. Moreover, the set $\C$ is (weak$^\ast$-)compact and the function $v^\ast \to  \Phi_{v^\ast} (x,z) $ is continuous for fixed $(x,z)$. Thus, the result follows from \cite[Theorem 3.1]{Perez-Vilches2021}.
\smartqed
\end{proof}

\section{Variational convergence of $\varphi_\lambda$}
\label{sec:varcvg}
 
In this section, we show that our inner regularization of the probability function \eqref{Proba:funct} inherits similar variational properties from the Moreau envelope (see Proposition \ref{basicprop}).
 
\begin{theorem}\label{var_convergence}
Let $\mathcal{H}$ be a separable Hilbert space, $\xi : \Omega \to \R^m$ be a random vector having density with respect to the Lebesgue measure, 
$\K\subset \Y$ be a (nonempty) convex cone of a (possible nonseparable) Banach space and  $\Phi :\mathcal{H}\times  \mathbb{R}^m \to \Y$ be  a continuous function such that \eqref{scalar01} holds. Then, the probability function $\varphi$ given in \eqref{Proba:funct} and the regularization $\varphi_\lambda$ given in \eqref{Proba:funct:lamb} satisfy the following properties:
\begin{enumerate}[label=\alph*)]
	\item For all $\lambda_1 > \lambda_2 >0$, $	\varphi_{\lambda_1} (x) \geq 	\varphi_{\lambda_2} (x)$ and $\inf\limits_{\lambda >0} \varphi_\lambda(x)= \varphi(x)$ for all $x \in \mathcal{H}$.

\item For any sequence $\lambda_{k} \to 0$ and $x_k \to x$ we have that 
\begin{align}\label{upper:ineq}
	\limsup\limits_{k \to \infty} \varphi_{\lambda_k} (x_k)  \leq \varphi(x).	
\end{align}
Furthermore, if the function $h$ from \eqref{scalar01} is sequentially weakly continuous on $\mathcal H$, then for any sequence $\lambda_{k} \to 0$ and $x_k \rightharpoonup x$ we have that  \eqref{upper:ineq} also holds.

\item For any sequence $\lambda_{k} \to 0$ and any sequence $x_k \to x \in \D$, we have
\begin{align}\label{equality:limits}
	\lim\limits_{k \to \infty} \varphi_{\lambda_k} (x_k) = \varphi(x).	
\end{align}
where $\D$ is the open set $\D:=\{ x\in \mathcal{H}:\exists z \text{ s.t. } S^h_\Phi(x,z) < h(x)\}$.
\item The functions $\varphi$ and $\varphi_\lambda$ are sequentially weakly upper semicontinuous on $\mathcal{H}$. 
\item The functions $\varphi$ and $\varphi_\lambda$ are continuous on $\D$.
\end{enumerate}
 \end{theorem}
 
 \begin{proof} 
 To prove a), let $\lambda_1 > \lambda_2 >0$. Then, for any fixed $(x,z)$ we have the inequalities $\PMoreau{\Phi}{\lambda_1}{h}(x,z) \leq \PMoreau{\Phi}{\lambda_2}{h}(x,z) \leq S^h_\Phi(x,z )$. Hence, $\varphi_{\lambda_1} (x) \geq 	\varphi_{\lambda_2} (x) \geq \varphi(x)$ and, by virtue of Proposition \ref{basicprop}, item a), it follows that $$\lim\limits_{\lambda\searrow 0} \varphi_\lambda(x)=\inf\limits_{\lambda >0} \varphi_\lambda(x)\geq \varphi(x).$$
	 Let $\lambda_{k} \to 0$ and fix $x\in \mathcal{H}$. By Proposition \ref{basicprop}, item a) we have that
	 \begin{align*}
	   \lim\limits_{k\to \infty} \PMoreau{\Phi}{\lambda_k}{h}(x,z) &=  S^h_\Phi (x,z ), \text{ for all } z\in \mathbb{R}^m.
	 \end{align*}
	 Thus, for all $z\in \mathbb{R}^m$, $ \liminf\limits_{ k \to \infty} \mathds{1}_{\hat{A}_k}( z) \geq \mathds{1}_{\hat{A}}(z)$, where 
	  \begin{equation*}
	 	\hat{A}_k := \{ z\in \mathbb{R}^m : \PMoreau{\Phi}{\lambda_k}{h}(x,z)> h(x)\},\,	\hat{A} :=  \{ z\in \mathbb{R}^m : S^h_\Phi (x,z ) > h(x)\}.
	 \end{equation*}
	
Then, by using the fact that $\xi$ has a density with respect to the Lebesgue measure and applying Fatou's Lemma, we get 
	 \begin{align*}
	 	1- \lim\limits_{k\to \infty} \varphi_{\lambda_{k}} (x)=\lim\limits_{k\to \infty} \mathbb{P}( \xi^{-1}(\hat{A}_k))&=\liminf\limits_{k\to \infty} \mathbb{P}( \xi^{-1}(\hat{A}_k)) \\
	 	&\geq \mathbb{P}( \xi^{-1}(\hat{A})) =1- \varphi(x).
	 \end{align*}
	 Therefore, $\lim\limits_{k\rightarrow\infty} \varphi_{\lambda_k}(x)\leq \varphi(x)$, which concludes the proof of a).
	 
	 To prove b), consider $\lambda_{k} \to 0$, $x_k \to x$ ($x_k \rightharpoonup x$, respectively) and the sets
	 \begin{align*}
	 	A_k& := \{ z\in \mathbb{R}^m : \PMoreau{\Phi}{\lambda_k}{h}(x_k,z)> h(x_k) \},\,	A :=  \{ z\in \mathbb{R}^m : S^h_\Phi (x,z ) > h(x) \}.
	 \end{align*}
	 Now, due to Proposition \ref{basicprop}, item c) and the continuity of $h$ (sequentially weak continuity of $h$ on $\mathcal{H}$, respectively) we get for any $z \in \Re^m$ that 
	 \begin{equation*}
	   \liminf\limits_{k\to \infty} \left( \PMoreau{\Phi}{\lambda_k}{h}(x_k,z)-h(x_k) \right) \geq  S^h_\Phi (x,z )-h(x),
	 \end{equation*}
	 which implies
	 \begin{align*}
	 	\liminf\limits_{ k \to \infty} \mathds{1}_{A_k}( z) \geq \mathds{1}_A(z), \text{ for all } z\in \mathbb{R}^m.
	 \end{align*}
	 Then using again the fact that $\xi$ has a density   and applying Fatou's Lemma we get  
	 \begin{align*}
	 	1- \limsup\limits_{k\to \infty} \varphi_{\lambda_{k}} (x_k)=\liminf\limits_{k\to \infty} \mathbb{P}( \xi^{-1}(A_k)) \geq \mathbb{P}( \xi^{-1}(A)) =1- \varphi(x), 
	 \end{align*}
	 which proves \eqref{upper:ineq}.
	 
	 Now, let us show c). Assume that $x_k \to x$, so by Proposition \ref{basicprop}, item d) and the continuity of $h$ we have that 
	  \begin{align}\label{eqlim001}
	    \lim_{k\to \infty} (\PMoreau{\Phi}{\lambda_k}{h}(x_k,z)-h(x_k)) =  S^h_\Phi (x,z )-h(x)  \text{ for all } z \in \mathbb{R}^m.
	  \end{align} Hence, by using the sets $A_k$ and $A$ defined above and by similar arguments as before, we obtain
	  \begin{equation}\label{upper:ineq2}
	    \limsup\limits_{k \to \infty} \varphi_{\lambda_k} (x_k)  \leq \varphi(x).
	  \end{equation}
   On the other hand, we consider the sets
	  \begin{align*}
	 	B_k& := \{ z\in \mathbb{R}^m : \MoreauYosida{\Phi}{\lambda_k}(x_k,z)<h(x_k) \},\,	B :=  \{ z\in \mathbb{R}^m : S^h_\Phi (x,z ) < h(x) \}.
	  \end{align*}
	 Then, mimic the last proof, we obtain that
	 \begin{align}\label{ineq:001}
	  \liminf\limits_{k\to \infty} \mathbb{P}( \xi^{-1}(B_k)) \geq \mathbb{P}( \xi^{-1}(B)). 
	 \end{align}
	 Since $x\in \D$ (and recalling that $\xi $ has density), we have that $\mathbb{P}( S^h_\Phi (x,\xi) =h(x))=0$. Hence, by using \eqref{upper:ineq2} and  \eqref{ineq:001}, it follows that
	 \begin{align*}
	 		\limsup\limits_{k \to \infty} \varphi_{\lambda_k} (x_k)  \leq \varphi(x) = \mathbb{P}( \xi^{-1}(B))  \leq   \liminf\limits_{k\to \infty} \mathbb{P}( \xi^{-1}(B_k)) \leq   \liminf\limits_{k\to \infty} \varphi(x_k),
	 \end{align*}
	 which completes the proof of \eqref{equality:limits}.
	 
	 To prove d), we consider $x_n\rightharpoonup x$ and the sets 
	 \begin{align*}
	 	C_n& := \{ z\in \mathbb{R}^m : S^h_\Phi (x_n,z ) >h(x_n) \},\,	C :=  \{ z\in \mathbb{R}^m : S^h_\Phi (x,z ) > h(x) \}.
	  \end{align*}
	 From the weak lower semicontinuity of $S^h_\Phi$ and the sequentially weak continuity of $h$, 
	 $$\liminf\limits_{n\to \infty}\left( S^h_\Phi (x_n,z )-h(x_n) \right)\geq S^h_\Phi (x,z)-h(x).$$ 
	 Hence,  following  an analogous argumentation, we can conclude that
	 \begin{equation*}
	   \liminf\limits_{n\to \infty} \mathbb{P}( \xi^{-1}(C_n)) \geq \mathbb{P}( \xi^{-1}(C)).
	 \end{equation*}
	 Thus,   applying Fatou's Lemma, we get  
	 \begin{align*}
	 	1- \limsup\limits_{n\to \infty} \varphi (x_n)=\liminf\limits_{n\to \infty} \mathbb{P}( \xi^{-1}(C_n)) \geq \mathbb{P}( \xi^{-1}(C)) =1- \varphi(x).
	 \end{align*}
	 Therefore, $\limsup\limits_{n\to \infty}\varphi(x_n)\leq \varphi(x)$. Now for a fixed $\lambda > 0$, the upper semicontinuity of $\varphi_{\lambda}$ follows from similar arguments as before but upon considering the sets
	 \begin{align*}
	 	\hat{C}_n& := \{ z\in \mathbb{R}^m : \PMoreau{\Phi}{\lambda}{h}(x_n,z)> h(x_n) \},\,	\hat{C} :=  \{ z\in \mathbb{R}^m : \PMoreau{\Phi}{\lambda}{h}(x,z)> h(x)\},
	  \end{align*}
	 
 	 Finally, let us prove e). Assume that $x_n \to x$. By the continuity of $S^h_\Phi$ and continuity of $h$, we have \eqref{eqlim001} holds. Thus, by using, once again, similar arguments but with the sets
	  \begin{align*}
	 	D_n& := \{ z\in \mathbb{R}^m : S^h_\Phi (x_n,z )<h(x_n) \},\,	D :=  \{ z\in \mathbb{R}^m : S^h_\Phi (x,z ) < h(x) \},
	  \end{align*}
	 we get that
	 \begin{align}\label{ineq:002}
	  \liminf\limits_{n\to \infty} \mathbb{P}( \xi^{-1}(D_n)) \geq \mathbb{P}( \xi^{-1}(D)). 
	 \end{align}
	 Now, since $x\in \D$  we have that $\mathbb{P}( S^h_\Phi (x,\xi) =h(x))=0$. Hence, by using part d) and  \eqref{ineq:002}, we have
	 \begin{align*}
	 		\limsup\limits_{n \to \infty} \varphi(x_n)  \leq \varphi(x) = \mathbb{P}( \xi^{-1}(D))  \leq   \liminf\limits_{n\to \infty} \mathbb{P}( \xi^{-1}(D_n)) \leq   \liminf\limits_{n\to \infty} \varphi(x_n),
	 \end{align*}
	 which yields the continuity of $\varphi$. The continuity of $\varphi_{\lambda}$ follows from similar arguments.
\smartqed
\end{proof}

\begin{remark}[Slater condition for $S_\Phi^h$]
It is worth mentioning that, in order to have the existence of a point $(x,z)$ such that $\S_\Phi^h(x,z) < h(x)$, the set $\C$ cannot contain the zero vector. Indeed, if $0 \in \C$, then from \eqref{supfunctionS} it follows that $\S_\Phi^h(x,z) \geq h(x)$, for all $(x,z) \in\mathcal{H}\times \mathbb{R}^m$.
On the other hand, if $\C$ is such that $\inf\{ \| v^\ast \| : v^\ast \in \C\} >0$ and $(x,z)$ satisfy $\Phi(x,z) \in \interior (-\K)$, then $\langle w^*, \Phi(x,z) \rangle < -\eta\| v^\ast\|$ holds for all $v^\ast \in \C$ with some $\eta >0$, and, thus, $S_\Phi^h(x,z) < h(x)$.
\end{remark}

Now, we formally describe the convergence properties of the family $\varphi_\lambda$ to the function $\varphi$ in terms of hypo-convergence.
\begin{corollary}
  Under the assumptions of Theorem \ref{var_convergence}, the sequence of regularizations $\varphi_\lambda$ hypo-converges to the probability function $\varphi$. In addition, suppose that the function $h$ in \eqref{scalar01} is weakly continuous, then the sequence of regularizations $\varphi_\lambda$  Mosco hypo-converges to the probability function $\varphi$.
\end{corollary}

\begin{proof}
The constant sequence $x_k=x$ together with the pointwise convergence in Theorem \ref{var_convergence} Item a) gives us the existence of a sequence $x_k\rightarrow x$ such that
\begin{equation*}
   \liminf_{k\rightarrow \infty}\varphi_{\lambda_k}(x_k)\geq \varphi(x).
\end{equation*}
The remaining second condition to obtain hypo-convergence follows from \eqref{upper:ineq2} obtained in the proof of Theorem \ref{var_convergence} Item c). If we suppose that the function $h$ in \eqref{scalar01} is weakly continuous, then the remaining second condition to obtain Mosco hypo-convergence is given by Theorem \ref{var_convergence} Item b). 
\smartqed
\end{proof}

\section{Differentiability and gradient formula for $\varphi_\lambda$}
\label{sec:diff}

In this section, we assume that $\mathcal{H}$ is finite-dimensional. Here we apply the results of subsection \ref{subsec:spherical_descomposition_gradient_formula} to give a formula for the gradients of our inner regularization of the probability function \eqref{Proba:funct}, and later we provide the  consistency of the gradients of our inner regularization. 

First we provide the following lemma, which shows that the gradients of $\PMoreau{\Phi}{\lambda}{h}-h$ satisfies
a growth condition.
\begin{lemma}
Let $\lambda > 0$ be given but fixed. Let $\bar x$ be a point such that $S^h_\Phi(\bar x,0) <h(\bar{x})$. Then, there exists $C_{\lambda}, \varepsilon >0$ such that 
\begin{equation}\label{thetagrowth20}
	\| \nabla_x \PMoreau{\Phi}{\lambda}{h}(x,z)-\nabla h(x) \| \leq C_{\lambda} (\| z\|+1),  \;
	\text{ for all } x \in \mathbb{B}_{\epsilon} (\bar{x})  \text{ and all } \; z\in \R^m.  
\end{equation}
\end{lemma}

\begin{proof}
We have, by \eqref{eq:derivativeofMoreau} and the triangle inequality, that
\begin{equation*}
\| \nabla_x \PMoreau{\Phi}{\lambda}{h}(x,z)- \nabla h(x)\| \leq \frac{1}{\lambda}\left( \Vert x\Vert +\Vert \Prox{S^h_\Phi}{\lambda}(x,z)\Vert \right)+\Vert\nabla h(x)\Vert,
\end{equation*}
for all $x\in \mathcal{H}$ and all $z\in \mathbb{R}^m$. By the nonexpansiveness of the proximal mapping we get
\begin{equation*}
\| \nabla_x \PMoreau{\Phi}{\lambda}{h}(x,z)- \nabla h(x)\| \leq \frac{1}{\lambda}\left( 2\Vert x\Vert+\Vert z\Vert +\Vert \Prox{S^h_\Phi}{\lambda}(0,0)\Vert \right)+\Vert\nabla h(x)\Vert,
\end{equation*}
for all $x\in \mathcal{H}$ and all $z\in \mathbb{R}^m$. Since $\nabla h$ is locally bounded at $\bar{x}$ ($h$ is continuously differentiable), there exists $\epsilon>0$ and $M>0$ such that
\begin{equation*}
\| \nabla_x \PMoreau{\Phi}{\lambda}{h}(x,z)- \nabla h(x)\| \leq \frac{1}{\lambda}\left( 2\epsilon+2\Vert \bar{x}\Vert+\Vert z\Vert +\Vert \Prox{S^h_\Phi}{\lambda}(0,0)\Vert \right)+M,
\end{equation*}
for all $x\in \mathbb{B}_\epsilon(\bar{x})$ and all $z\in \mathbb{R}^m$. We conclude by defining
\begin{equation*}
  C_{\lambda} :=\max\{\frac{1}{\lambda}\left( 2\epsilon+2\Vert \bar{x}\Vert +\Vert \Prox{S^h_\Phi}{\lambda}(0,0)\Vert \right)+M,\frac{1}{\lambda}\}.
\end{equation*}
\smartqed
\end{proof}

In order to apply the gradient formula given in Theorem \ref{thm:clarkeappli} it will be convenient to introduce the following notation. Given a parameter $\lambda >0$, let us assume $x$ belonging to an appropriate neighbourhood of $\bar x$ such that $\PMoreau{\Phi}{\lambda}{h} (x,0) < h(x)$. Then, we define the set of finite and infinite directions for the function $\PMoreau{\Phi}{\lambda}{h}$ by 
\begin{equation}\label{flambda}
  F_\lambda(x):= \{ v\in \mathbb{S}^{m-1} : \exists r>0: \PMoreau{\Phi}{\lambda}{h} (x,rLv)=h(x) \}, \quad I_\lambda(x):=\mathbb{S}^{m-1}\backslash F_\lambda(x),
\end{equation}
respectively, and its associated radial function  given by 
\begin{equation}\label{defradialfunction}
  \rho_\lambda \left( x,v\right) :=\sup\left\{ r>0 : \PMoreau{\Phi}{\lambda}{h}(x,rLv)\leq\* h(x) \right\}
\end{equation}

\begin{remark}[Characterization of radial function]\label{remarkCharc}
It is important to recall that when $ \PMoreau{\Phi}{\lambda}{h} (x,0)< h(x)$, then the radial function $\rho_\lambda(x,v)$ can be characterized as the $\rho_\lambda(x,v) =\inf\{ r>0 :   \PMoreau{\Phi}{\lambda}{h} (x,rLv)>h(x) \}$,  with the convention $\inf \emptyset = +\infty$. Furthermore, it also can be characterized by \emph{unique solution} of the equation
\begin{align*}
   \PMoreau{\Phi}{\lambda}{h} (x,rLv)=h(x) 
\end{align*}
for any finite direction $v\in F_\lambda(x)$. We refer to \cite[Proposition 2.6]{MR4000225} for more details of the proof, which uses essentially the convexity and continuity. Nevertheless, it is clear that the continuity of the convex function is necessary, as was illustrated in \cite[Example 2.7]{MR4000225}.
\end{remark}

Finally, let us introduce the gradient-like mapping $G_\lambda : \mathcal{H}\times  \mathbb{S}^{m-1} \to  \mathcal{H}$ defined as 
\begin{equation}\label{def:GradientGe}
   \G{\lambda}(x,v)  := \left\{ \begin{array}{cc}
    - \theta(\rho_\lambda(x,v) ,v) \left( \frac{ \nabla_x \PMoreau{\Phi}{\lambda}{h} (x,\rho_\lambda(x,v)Lv)-\nabla h(x)}{ \left\langle  \nabla_z \PMoreau{\Phi}{\lambda}{h} (x,\rho_\lambda \left( x,v\right) Lv,Lv\right\rangle } \right)  &  \text{ if } v\in F_\lambda(x)\\ & \\
   0  & \text{ if } v\in I_\lambda(x)
   \end{array}  \right. 
\end{equation}
and the factor $\theta$ is defined in \eqref{transformationftorho}. Using the above notation we are able to provide a gradient formula for the probability function $\varphi_\lambda$

\begin{theorem}\label{Theo:form:gradient}
Let $\bar{x} \in \mathcal{H}$ be such that $S^h_\Phi(\bar{x},0) <h(\bar{x})$, and assume that $f_\xi$ 
	 satisfies the following growth
	 condition 	\begin{equation}
	 	\lim\limits_{\substack{\| z\|  \to +\infty } }  \| z\|^{m+1}  {f}_{\xi} (z) =0. \label{eq:assgrowth}
	 \end{equation}
	 
Then, for any given $\lambda > 0$, the probability function $\varphi_\lambda$, defined in \eqref{Proba:funct:lamb}, is continuously differentiable on an appropriate neighbourhood $U$ of
	 $\bar x$ and it holds:
	 \begin{equation*}
	 	\nabla \varphi_\lambda (x)=
	 	\int\limits_{ \mathbb{S}^{m-1}} \G{\lambda}(x,v)      d\mu_\zeta(v), \:\text{ for all } x\in U,
	 \end{equation*}
where $G_{\lambda}$ is as in \eqref{def:GradientGe}. Moreover, the gradients of $\PMoreau{\Phi}{\lambda}{h} $ can be computed by the formula
	 \begin{align}\label{compgrad}
	 \nabla 	\PMoreau{\Phi}{\lambda}{h} (x,z) &= \frac{ (x,z)- \Prox{(\langle v^\ast,\Phi \rangle+h)}{\lambda}(x,z) }{\lambda}, 
\end{align}
where $v^\ast $ is any active vector at $(x,z)$, that is, $ v^\ast \in\C$ and $ \MoreauYosida{\Phi^h_{v^\ast}}{\lambda}(x,z) =\PMoreau{\Phi}{\lambda}{h}(x,z)$ in view of Proposition \ref{prop:moreaumaxismaxmoreau}.
\end{theorem}
\begin{proof}
Let $\lambda > 0$ be given but fixed. Due to Proposition \ref{basicprop}, $\MoreauYosida{\Phi^h}{\lambda}\leq S^h_\Phi$. Thus, $\MoreauYosida{\Phi^h}{\lambda}(\bar{x},0)<h(\bar{x})$. Therefore, due to continuity, we can set aside an appropriate neighbourhood $U$ of $\bar x$ on which this continues to hold and on which the objects in equations \eqref{flambda}, \eqref{defradialfunction} and \eqref{def:GradientGe} are well defined. This neighbourhood can be taken independently of $\lambda > 0$.

Thus to prove the first part, by Theorem \ref{thm:clarkeappli}, it is enough to prove the $\eta_\theta$-growth condition. To this end let us pick an abitrary $\bar v \in I_{\lambda}(\bar x)$. 
In view of \eqref{thetagrowth20}, we define
\begin{equation*}
   \eta_{\theta}(r,v):=C_{\lambda}\left( r\Vert Lv\Vert+1 \right),
\end{equation*}
choose $l\geq 1/\epsilon$ and notice, by \eqref{transformationftorho}, that 
\begin{align*}
  r\theta(r,v)\eta_{\theta}(r,v)
  &=\frac{2\pi^{\frac{m}{2}} |\textnormal{det}(L)|
	}{\Gamma(\frac{m}{2})} C_{\lambda}(r^{m+1}\Vert Lv\Vert f_\xi (rLv)+r^{m}f_\xi (rLv))\\
	&=\frac{2\pi^{\frac{m}{2}} |\textnormal{det}(L)|
	}{\Gamma(\frac{m}{2})} C_{\lambda} \Vert r Lv\Vert^{m+1}f_\xi (rLv) \left ( \frac{1}{\Vert Lv\Vert^{m}} + \frac{1}{r \Vert Lv\Vert^{m+1}}  \right ) 
	\xrightarrow[\substack{r \to +\infty \\
				v \to \bar{v} }]{}0,
\end{align*}
where the last limit follows from assumption \eqref{eq:assgrowth}. Therefore, as a result of \eqref{thetagrowth}, the $\eta_\theta$-growth condition is satisfied. The computation for the gradient \eqref{compgrad} follows from \cite[Theorem 3.5]{Perez-Vilches2021}.
\smartqed
\end{proof}
The so-called radial function $\rho_\lambda$ is used in the last gradient formula. The following proposition shows that this mapping is continuous on the three parameters $(\lambda,x,v)$, which is a key property for numerical computations, and to provide the asymptotic behavior of the gradients of the probability function $\varphi_\lambda$ to the (sub-)gradients to the nominal function $\varphi$.

\begin{proposition}\label{continuity_rho}
Let us consider the radial function in \eqref{defradialfunction} and the open set defined by $U:=\{ x \in \mathcal{H}: S_\Phi^h(x,0) < h(x)\} $. Then, for every sequence $(\lambda_k,x_k,v_k) \to (\lambda,x,v) \in [0,+\infty) \times U \times \mathbb{S}^{m-1}$ we have that $\rho_{\lambda_k}(x_k,v_k) \to \rho_{\lambda}(x,v)$, where $\rho_0$ is defined by $ \rho_0(x,v):=  \sup\left\{ r>0 : S_\Phi(x,rLv)\leq h(x) \right\}$. 
\end{proposition}

\begin{proof}
Let us focus on the case $(\lambda_k,x_k,v_k) \to (0,x,v)$ since the proof for $\lambda>0$ is analogous. Let us first assume that the sequence $\{\rho_{\lambda_k}(x_k,v_k)\}$ admits a cluster point called $r$. Then for some subsequence of $\{\rho_{\lambda_k}(x_k,v_k)\}$ we have $\rho_{\lambda_k}(x_{k_l},v_{k_l})\to_l r$. By Proposition \ref{basicprop} Item d), continuity of $h$ and again by the characterization of the radial function as unique solution (see Remark \ref{remarkCharc}) we have that
\begin{equation*}
  0=\PMoreau{\Phi}{\lambda_{k_l}}{h}(x_{k_l},\rho_{\lambda_k}(x_{k_l},v_{k_l})Lv_{k_l})-h(x_{k_l})\to_l S_\Phi^h(x,rLv)-h(x), 
\end{equation*}
then  $r=\rho_0(x,v)$. Since this holds true for all possible cluster points, we have in fact that $\rho_{\lambda_k}(x_k,v_k)$ converges to $\rho(x,v)$, whenever the sequence $(\rho_{\lambda_k}(x_k,v_k))_{k\in \mathbb{N}}$ has a cluster point.

Next let us assume that, $\rho_{\lambda_k}(x_k,v_k) \to +\infty$, and by contradiction suppose that $r:=\rho(x,v) <+\infty$. 

Then, by Proposition \ref{basicprop} Item d), we have that for all large enough $k$ 
$$ 0< \PMoreau{\Phi}{\lambda_{k}}{h}(x_{k},(r+1)Lv_{k})-h(x_{k}), $$
which implies that $\rho_{\lambda_k} (x_{k},v_k) < (r+1)$ for all large enough $k$ (see Remark \ref{remarkCharc}), which contradicts our assumption, and concludes the proof. 
\smartqed
\end{proof}

The next proposition shows that the radial function $\rho_\lambda$, given in \eqref{defradialfunction}, can be computed using the associated radial function to the function $\MoreauYosida{\Phi^h_{v^\ast}}{\lambda}$, defined in \eqref{scalar01}, that is, for a given $v^\ast \in \C$, and $\lambda>0$ we set
\begin{align*}
   \rho_\lambda^{v^\ast} (x,v) := \sup\{ r>0 : \MoreauYosida{\Phi_{v^\ast}^h }{\lambda}(x,rLv) \leq h(x)\}.
\end{align*}
\begin{proposition}
In the setting of Proposition \ref{continuity_rho}, we have that 
\begin{align*}
   \rho_\lambda (x,v) :=\min\{  \rho_\lambda^{v^\ast} (x,v) : v^\ast \in \C\}.
\end{align*}
\end{proposition}
\begin{proof}
The proof follows the same lines of arguments that \cite[Proposition 2.6]{vanAckooij_Perez-Aros_2019}, which only uses the supremum structure of the function. 
\smartqed
\end{proof}

Now, we focus on well-possness of the gradient approximation, that is,  the study of convergence properties of the gradients of the regularized probability functions $\varphi_\lambda$. Since, the probability function $\varphi$ is not necessarily smooth, it is necessary to introduce some terminology from generalized differentiation theory.

Let us recall that for a given function $f:\mathcal{H} \to\mathbb{R} $, the set
\begin{align*}
   \pfrech f(x):=& \bigg\{ x^*\in \mathcal{H} : \liminf\limits_{h \to 0} \frac{f(x+h) -f(x) - \langle x^*, h\rangle}{\| h\|} \geq 0 \bigg\}
\end{align*}
is called the  regular  subdifferential  of $f$ at $x$. 
The basic subdifferential can
be defined as (see e.g. \cite{MR1491362,MR2191744,MR3823783})
\begin{align*}
\pmord f(x)& := \{ x^\ast \in \mathcal{H}  : x_k^* \in \pfrech f(x_k), \text{ and } (x_k,f(x_k),x_k^\ast ) \to (x,f(x),x^\ast )   \},
\end{align*}
 The following proposition provides a (sub-)differential variational principle for the probability function $\varphi$ using the inner regularized functions $\varphi_\lambda$.
\begin{proposition}\label{01_estimation}
Under the assumption of Theorem \ref{Theo:form:gradient} we have that for every $x^\ast \in  \pfrech \varphi(\bar x) $ and every $\epsilon >0$ there exists $\lambda >0$, $x_\lambda \in \mathcal{H}$ such that $\| \bar x - x_\lambda\| + \| x^\ast - \nabla \varphi_\lambda (x_\lambda) \| + |\varphi(\bar x) - \varphi(x_\lambda)| \leq \epsilon $. Particularly, we have that 
$ \pmord \varphi(\bar x) \subseteq \limsup_{x \to \bar{x},\; \lambda \to 0^+} \{ \nabla \varphi_\lambda(x) \}$.
\end{proposition}

\begin{proof}
The first part follows from a direct application of \cite[Lemma 2.1]{MR4000225}. For the second part, consider a point $x^\ast  \in \pmord \varphi(\bar x)$, by definition there are $x_k^\ast \in \pfrech \varphi(x_k)$ with $x_k \to \bar x$, $\varphi(x_k) \to \varphi(\bar x)$ and $x_k^\ast \to x^\ast$. By the last part applied to $x_k$ (for large enough $k$) we have that each $x^\ast_k$ can be approximated by gradients of the probability functions $\varphi_\lambda$, which by a classical diagonal argument shows the desire inclusion.
\smartqed
\end{proof}
The last result shows that the basic subdifferential of the probability function $\varphi$ can be upper-estimated by using the gradients of the probability function $\varphi_\lambda$. In the rest of this subsection, we will focus on providing the opposite inclusion, that is to say,  the accumulations points of gradients are points in the basic subdifferential.  
\begin{lemma}\label{Prox_growth_condition}
Let us suppose the mapping $S_\Phi^h $ defined in \eqref{supfunctionS} is bounded from below by an affine linear function $\bar{h}:\mathcal{H}\to\R$, let $\bar{x}\in \mathcal{H}$ such that $S_{\Phi}^h(\bar{x},0) <h(\bar{x})$. Given $\epsilon>0$, there exists $\lambda_0,\epsilon_0 >0$ such that for all $(\lambda,x,v)\in (0,\lambda_0)\times \mathbb{B}_{\epsilon_0 }(\bar{x})\times\mathbb{S}^{m-1}$ with $v\in F_\lambda(x)$
\begin{equation}\label{Inq}
\| x - \hat{x}\| +  \|\rho_{\lambda}(x,v) Lv - \hat{z}     \|\leq  \epsilon,
\end{equation}
where $(\hat{x},\hat{z}):=\Prox{S_\Phi^h}{\lambda}{(x,\rho_{\lambda}(x,v)Lv)}$.
\end{lemma}

\begin{proof}
Let $\epsilon_0 \in (0,\epsilon)$ such that $\mathbb{B}_{\epsilon_0} (\bar{x})  \subset U$, and pick $(\lambda,x,v)$ with $x \in \mathbb{B}_{\epsilon_0} (\bar{x})$, $\lambda \in (0,1)$ and $v\in F_\lambda(x)$.  We first notice that
\begin{equation*}
  S_\Phi^h(\hat x,\hat z)+\tfrac{1}{2\lambda}\|\hat{x}-x\|^2+\tfrac{1}{2\lambda}\|\hat{z}-z\|^2= \PMoreau{\Phi}{\lambda}{h}(x,z) = h(x)
\end{equation*}
where $z:=\rho_{\lambda}(x,v)Lv$ and the last equality follows from the definition of the latter term. Now, let us suppose that $\bar{h} = \langle x^\ast , \cdot\rangle +\beta$, so
\begin{equation}\label{lemma4.7_eq1}
   \langle x^\ast,\hat{x}\rangle+\beta+\tfrac{1}{2\lambda}\|\hat{x}-x\|^2+\tfrac{1}{2\lambda}\|\hat{z}-z\|^2\leq h(x).
\end{equation}
On the other hand, the inequality
\begin{equation*}
  |\langle x^\ast,\hat{x}-x\rangle|\leq \tfrac{1}{2}\|x^\ast\|^2+\tfrac{1}{2}\|\hat{x}-x\|^2
\end{equation*}
implies
\begin{equation}\label{lemma4.7_eq2}
  |\langle x^\ast,\hat{x}\rangle|\leq \tfrac{1}{2}\|x^\ast\|^2+\tfrac{1}{2}\|\hat{x}-x\|^2+\tfrac{1}{2}\|x^\ast\|^2+\tfrac{1}{2}\|x\|^2.
\end{equation}
From \eqref{lemma4.7_eq1} and \eqref{lemma4.7_eq2} we have that
\begin{equation*}
   \left(\tfrac{1}{2\lambda}-\tfrac{1}{2}\right)\|\hat{x}-x\|^2+\tfrac{1}{2\lambda}\|\hat{z}-z\|^2\leq h(x)-\beta+\|x^\ast\|^2+\tfrac{1}{2}\|x\|^2. 
\end{equation*}
Since $\frac{1}{1-\lambda} > 1$, and due to continuity of $h$, a constant $M > 0$ such that 
\begin{equation*}
   \|\hat{x}-x\|^2+\|\hat{z}-z\|^2\leq
   \left(\frac{2\lambda }{1-\lambda }\right)M. 
\end{equation*}

Now, considering $\lambda_0>0$ small enough, we can conclude that 
\begin{equation*}
   \|\hat{x}-x\|^2+\|\hat{z}-z\|^2\leq \epsilon^2, 
\end{equation*}
for all $\lambda \in (0,\lambda_0) $ and $x\in \mathbb{B}_{\epsilon_0} (\bar{x}) $, which shows that   \eqref{Inq} holds.
\smartqed
\end{proof}

In the following lemma we will require that the mapping $S_\Phi^h $ definied in \eqref{supfunctionS} satisfies the following growth condition at $\bar{x}$: there exist constants $\epsilon,\ell>0$ such that 
\begin{align}\label{growth}
\| \partial_x S_\Phi^h(x,z) \| \leq  \eta(\|z\|),  \;
		\forall x \in \mathbb{B}_{\epsilon} (\bar{x}),\; \forall  \|z\| \geq \ell,  
\end{align}
for some nondecresing function $\eta$ satisfying 
\begin{equation*}
   \lim_{\|z\|\rightarrow \infty} \|z\|^{m}f_\xi(z)\eta(\|z\|+\alpha)=0
\end{equation*}
for some $\alpha>0$. Here, the norm of a sub-differential set is defined as follows:
\begin{align*}
\| \partial_x S^h_\Phi(x,z) \|:=\sup\{ \| x^\ast\| : x^\ast \in  \partial_x S^h_\Phi(x,z)  \}
\end{align*}

\begin{lemma}\label{Lemma_UB}
Let us suppose the mapping $S_\Phi^h $ definied in \eqref{supfunctionS} is bounded from below by an affine linear function $\bar{h}:\mathcal{H}\to\R$ and satisfies the growth condition \eqref{growth} at $\bar x$, where $S_\Phi^h(\bar{x},0) <h(\bar{x})$. Then there exits $\gamma >0$ and  $\kappa >0$ such that 
\begin{align}\label{Lemma_UB01}
   \| \G{\lambda}(x,v) \|  \leq \kappa,\quad \forall 
  (\lambda,x,v)\in (0,\gamma)\times \mathbb{B}_\gamma(\bar x)\times \mathbb{S}^{m-1}
\end{align}
where $\G{\lambda}$ is defined in \eqref{def:GradientGe}. Moreover, for all $v\in I(\bar{x})$ and all $\epsilon >0$ there exists $\gamma >0$ such that 
\begin{align}\label{Lemma_UB02}
   \| \G{\lambda}(x,v) \|  \leq \epsilon,\quad \forall 
  (\lambda,x,v)\in (0,\gamma)\times \mathbb{B}_\gamma(\bar x)\times \mathbb{B}_{\gamma} (\bar{v}).
\end{align}
\end{lemma}

\begin{proof}
First, let us show that for every $\bar{v}\in\mathbb{S}^{m-1}$ there exist $\epsilon_{\bar{v}}>0$ and $M_{\bar{v}}>0$ such that 
\begin{equation*}
   \|G_{\lambda}(x,v)\|\leq M_{\bar{v}}, \text{ for all } (\lambda,x,v)\in (0,\epsilon_{\bar{v}})\times \mathbb{B}_{\epsilon_{\bar{v}}}(\bar{x})\times \mathbb{B}_{\epsilon_{\bar{v}}}(\bar{v}).
\end{equation*}

First we notice that there exist $\epsilon_1>0$ and $\beta_1>0$ such that we have $S_\Phi^h({x},0)-h( x)\leq -\beta_1$ for all $x\in \mathbb{B}_{\epsilon_1}(\bar{x})$. Then, for all $v\in F_\lambda(x)$, (see, e.g., \cite[Lemma 2.1 item 2]{vanAckooij_Henrion_2016})
\begin{align*}
   \tfrac{-\rho_{\lambda}(x,v)}{2}\langle \nabla_z &\PMoreau{\Phi}{\lambda}{h}(x,\rho_{\lambda}(x,v)Lv),Lv\rangle\nonumber\\
  =
  &\langle \nabla_z \PMoreau{\Phi}{\lambda}{h}(x,\rho_{\lambda}(x,v)Lv), \tfrac{\rho_{\lambda}(x,v)}{2}Lv-\rho_{\lambda}(x,v)Lv\rangle\nonumber\\
   \leq& \PMoreau{\Phi}{\lambda}{h}(x,\tfrac{\rho_{\lambda}(x,v)}{2}Lv)-\PMoreau{\Phi}{\lambda}{h}(x,\rho_{\lambda}(x,v)Lv)\nonumber\\
  =& \PMoreau{\Phi}{\lambda}{h}(x,\tfrac{\rho_{\lambda}(x,v)}{2}Lv)-h(x)\nonumber\\
   \leq& \tfrac{1}{2}\PMoreau{\Phi}{\lambda}{h}(x,0)+\tfrac{1}{2}\PMoreau{\Phi}{\lambda}{h}(x,\rho_{\lambda}(x,v)Lv)-h(x)\nonumber\\
  =& \tfrac{1}{2}\PMoreau{\Phi}{\lambda}{h}(x,0)-\tfrac{1}{2}h(x)\nonumber\\
   \leq & \tfrac{1}{2}S_\Phi^h(x,0)-\tfrac{1}{2}h(x)\nonumber\\
   \leq&-\tfrac{\beta_1}{2}.
\end{align*}
Since $\nabla h$ is locally bounded we have
\begin{equation}\label{bound_G}
\begin{aligned}
   \|\G{\lambda}(x,v)\|
   \leq &\tfrac{2\pi^{m/2}\text{det}(L)}{\Gamma(m/2)\beta_1}\rho_{\lambda}(x,v)^{m}f_\xi(\rho_{\lambda}(x,v)Lv)\\
  &\times\left(\|\nabla_x \PMoreau{\Phi}{\lambda}{h}(x,\rho_{\lambda}(x,v)Lv)\|+\beta_2\right)
\end{aligned}
\end{equation}
for some $\beta_2>0$ and for all $x\in \mathbb{B}_{\epsilon_1}(\bar{x})$ and $v\in F_\lambda(x)$.

Now let $\bar{v}\in \mathbb{S}^{m-1}$ be fixed. If $\bar{v}\notin I(\bar{x})$
then there exist $\epsilon_{\bar{v}}>0$ such that $v\notin I_{\lambda}(x)$ for all $(\lambda,x,v)\in (0,\epsilon_{\bar{v}})\times \mathbb{B}_{\epsilon_{\bar{v}}}(\bar{x})\times \mathbb{B}_{\epsilon_{\bar{v}}}(\bar{v})$ where $I_{\lambda}(x)$ was defined in \eqref{flambda}. Indeed, if it is not true, then there exists a sequence $(\lambda_k,x_k,v_k)\to (0,\bar x,\bar v)$ with $v_k\in I_{\lambda_k}(x_k)$. Hence, $\rho_{\lambda_k}(x_k,v_k)=\infty$ and so $\rho_0(\bar x,\bar v)=\infty$ by Proposition \ref{continuity_rho}. This yields a contradiction with $\bar v\notin I(\bar x)$. 

Since $S_\Phi^h$ is continuous at $(\bar x,\bar{z})$, where $\bar{z}:=  \rho_0(\bar{x},\bar{v})L\bar{v}$, there exist $\epsilon_2>0$ and $\beta_3>0$ such that for all $(x,z)\in \mathbb{B}_{\epsilon_2} (\bar{x},\bar{z})$
\begin{equation}\label{subdiff_bounded}
   \|(u^\ast,v^\ast) \|\leq \beta_3,\text{ for all } (u^\ast,v^\ast)\in \partial S_\Phi^h(x,z).
\end{equation}
Now, by    Proposition \ref{continuity_rho} and Lemma \ref{Prox_growth_condition}, and considering $\epsilon_{\bar{v}}$ small enough, we get that
\begin{equation*}
   \|(\bar{x},\rho_0(\bar{x},\bar{v})L\bar{v})-\Prox{S_\Phi^h}{\lambda}{(\bar{x},\rho_0(\bar{x},\bar{v})L\bar{v})}\|\leq \tfrac{\epsilon_2}{2}
\end{equation*}
and
\begin{equation*}
   \|(\bar{x},\rho_0(\bar{x},\bar{v})L\bar{v})-(x,\rho_{\lambda}(x,v)Lv)\|\leq \tfrac{\epsilon_2}{2}
\end{equation*}
for all $(\lambda,x,v)\in (0,\epsilon_{\bar{v}})\times \mathbb{B}_{\epsilon_{\bar{v}}}(\bar{x})\times \mathbb{B}_{\epsilon_{\bar{v}}}(\bar{v})$. Thus, since the proximal mapping is nonexpansive we have that
\begin{align}
   \|(\bar{x},\rho_0(\bar{x},\bar{v})L\bar{v})-&\Prox{S_\Phi^h}{\lambda}{(x,\rho_{\lambda}(x,v)Lv)}\|\nonumber\\
   \leq& \|(\bar{x},\rho_0(\bar{x},\bar{v})L\bar{v})-\Prox{S_\Phi^h}{\lambda}{(\bar{x},\rho_0(\bar{x},\bar{v})L\bar{v})}\|\nonumber\\
  &+\|(\bar{x},\rho_0(\bar{x},\bar{v})L\bar{v})-(x,\rho_{\lambda}(x,v)Lv)\|\leq \epsilon_2 \label{Prox_inball}
\end{align}
for all $(\lambda,x,v)\in (0,\epsilon_{\bar{v}})\times \mathbb{B}_{\epsilon_{\bar{v}}}(\bar{x})\times \mathbb{B}_{\epsilon_{\bar{v}}}(\bar{v})$. Now, by \eqref{subdiff_bounded}, \eqref{Prox_inball} and since due to \eqref{eq:nablaprox}
\begin{equation*}
   \nabla \PMoreau{\Phi}{\lambda}{h}(x,\rho_{\lambda}(x,v)Lv)\in \partial S_\Phi^h(\Prox{S_\Phi^h}{\lambda}{(x,\rho_{\lambda}(x,v)Lv)})
\end{equation*}
we have that
\begin{equation}\label{bound_derivative}
   \|\nabla_x \PMoreau{\Phi}{\lambda}{h}(x,\rho_{\lambda}(x,v)Lv)\|\leq \|\nabla \PMoreau{\Phi}{\lambda}{h}(x,\rho_{\lambda}(x,v)Lv)\|\leq \beta_3
\end{equation}
for all $(\lambda,x,v)\in (0,\epsilon_{\bar{v}})\times \mathbb{B}_{\epsilon_{\bar{v}}}(\bar{x})\times \mathbb{B}_{\epsilon_{\bar{v}}}(\bar{v})$. By  \eqref{bound_derivative}, Proposition \ref{continuity_rho}, \eqref{bound_G} and considering $\epsilon_{\bar{v}}<\epsilon_1$ smaller if needed, we get that
\begin{equation*}
   \|G_{\lambda}(x,v)\|\leq M_1
\end{equation*}
for all $(\lambda,x,v)\in (0,\epsilon_{\bar{v}})\times \mathbb{B}_{\epsilon_{\bar{v}}}(\bar{x})\times \mathbb{B}_{\epsilon_{\bar{v}}}(\bar{v})$.

Now, let us assume $\bar{v}\in I(\bar{x})$ and consider $\gamma >0$. By the growth condition, we have that there exists $\ell,\epsilon >0$ such that 
\begin{equation}\label{growth_InqRho}
   \rho_{\lambda}(x,v)^{m}f_\xi(\rho_{\lambda}(x,v)Lv)\eta(\rho_{\lambda}(x,v)\|Lv\|+\alpha )\leq \gamma
\end{equation}
  whenever $\rho_{\lambda}(x,v)\|Lv\| \geq \ell $, and 
 \begin{align}\label{Eq00Lemma}
\| \partial_x S_\Phi^h(x,z) \| \leq  \eta(\|z\|),  \;
		\forall x \in \mathbb{B}_{\epsilon} (\bar{x}),\; \forall  \|z\| \geq \ell;  
\end{align} 

Now, by Lemma \ref{Prox_growth_condition}, we can consider $\epsilon_0,\lambda_0 >0$ such that $ \hat{x} \in \mathbb{B}_\epsilon(\bar{x})$ and 
\begin{align}\label{Eq01Lemma}
   \rho_\lambda (x,v)\| Lv\| + \alpha \geq \| \hat{z}\| \geq  \rho_\lambda (x,v)\| Lv\| - \alpha, 
\end{align}  for all $(\lambda, x, v) \in (0,\lambda_0) \times \mathbb{B}_{\epsilon_0}(\bar{x})  \times \mathbb{S}^{m-1}$ with $v\in F_\lambda(x)$, where $(\hat{x},\hat{z}):=\Prox{S_\Phi^h}{\lambda}{(x,\rho_{\lambda}(x,v)Lv)}$. Moreover, using Proposition \ref{continuity_rho}, when considering a small enough $\epsilon_3 \in ( 0,\min\{ \epsilon_0,\lambda_0\})$ it follows that:
\begin{equation}\label{Eq02Lemma}
   \rho_\lambda(x,v) \geq \frac{\ell +\alpha}{ \| Lv\|}, \text{ for all } (\lambda, x,v) \in (0,\epsilon_3) \times \mathbb{B}_{\epsilon_3} (\bar{x})  \times \mathbb{B}_{\epsilon_3}(\bar{v}).
\end{equation}
Now, mixing equations \eqref{Eq00Lemma}, \eqref{Eq01Lemma} and \eqref{Eq02Lemma}, we conclude that for all $(\lambda, x,v) \in (0,\epsilon_3) \times \mathbb{B}_{\epsilon_3} (\bar{x})  \times \mathbb{B}_{\epsilon_3}(\bar{v})$, we have
\begin{equation*}
   \|\nabla_x \PMoreau{\Phi}{\lambda}{h}(x,\rho_{\lambda}(x,v)Lv)\|\leq \eta( \|\hat{z}\|) \leq  \eta(\rho_\lambda(x,v)\|Lv\|+\alpha),
\end{equation*}
where we have used the fact that $\eta $ is nondecreasing and  $\nabla \PMoreau{\Phi}{\lambda}{h}(x,\rho_{\lambda}(x,v)Lv)$ belongs to the set $\partial S_\Phi^h(\Prox{S_\Phi^h}{\lambda}{(x,\rho_{\lambda}(x,v)Lv)})$. Then, replacing this into \eqref{bound_G}, and using  \eqref{growth_InqRho}, we get that 
\begin{align}\label{eq001inf}
\|G_\lambda(x,v)\| \leq \gamma, \text{ for all } (\lambda, x,v) \in (0,\epsilon_3) \times \mathbb{B}_{\epsilon_3} (\bar{x})  \times \mathbb{B}_{\epsilon_3}(\bar{v}).
\end{align}

Since $\mathbb{S}^{m-1}$ is compact and the family of neighborhoods $\mathbb{B}_{\epsilon_{\bar{v}}}(\bar v)$ covers $\mathbb{S}^{m-1}$, we can pick a finite subcover, that is, there exists $N\in\mathbb{N}$ and some $v_1,\ldots,v_N\in \mathbb{S}^{m-1}$ such that
\begin{equation*}
   \mathbb{S}^{m-1}\subset \bigcup_{i=1}^N \mathbb{B}_{\epsilon_{v_i}}(v_i).
\end{equation*}
Therefore, we choose $\gamma>0$ such that
\begin{equation*}
  (0,\gamma)\subset \min\{\epsilon_{v_i}:\; i=1\ldots,N\}\;\text{ and }\; \mathbb{B}_\gamma(\bar{x})\subset \bigcap_{i=1}^N \mathbb{B}_{\epsilon_{v_i}}(\bar{x})
\end{equation*}
and define $ \kappa:=\max\{M_{v_i}:\;i=1\ldots,N\}$ to conclude the proof of \eqref{Lemma_UB01}. Finally, the proof of \eqref{Lemma_UB02} follows from the more precise estimation  \eqref{eq001inf}.
\smartqed
\end{proof}

\begin{theorem}[Gradient Consistency]
  Let us suppose the mapping $S_\Phi^h $ defined in \eqref{supfunctionS} satisfies the $\eta_\theta$-growth condition at $(\bar x,\bar v) $ for all $\bar v\in I(x)$, and assume that $$\partial S_\Phi^h (\bar x, \rho (\bar x,v) Lv) \text{ is single valued for almost all } v\in \mathbb{S}^{m-1}.$$
  Then, the probability function $\varphi$, given in \eqref{Proba:funct} is Locally Lipschitzian at $\bar x$ and in fact Fréchet différentiable at $\bar x$. Moreover, any accumulation point of sequences $\{\nabla \varphi_{\lambda}(x_k)\}_{k \geq 0}$ with $\lambda_k \to 0^+$ and $x_k \to \bar{x}$ are equal to $\nabla\varphi(\bar{x})$.
\end{theorem}

\begin{proof}
First, let us notice that by Proposition \ref{01_estimation} we have that  for all $x$ close enough to $\bar{x}$
$$\pmord \varphi(x) \subset \limsup\limits_{\lambda \to 0^+ , x' \to x } \nabla   \varphi_\lambda(x'). $$
Now, by Lemma \ref{Lemma_UB} we have that the right-hand side set of the above inclusion is bounded, and consequently the function $\varphi$ is locally Lipschitz at $\bar{x}$ (see, \cite[Theorem 4.15]{MR3823783}). Then, due to  \cite[Theorem 4.17]{MR3823783} it is enough to show that 
$ \limsup_{\lambda \to 0^+ , x' \to \bar{x} }  \nabla  \varphi_\lambda(x') $ is single valued. Indeed, by Lemma \ref{Lemma_UB} we can apply Fatou's type theorem (see, e.g., \cite[Corollary~4.1]{MR2197293}) and obtain that 

\begin{equation}\label{inclusion_limsup_gradients}
    \limsup\limits_{\lambda \to 0^+ , x' \to \bar{x} }\nabla   \varphi_\lambda(x') \subset \int_{\mathbb{S}^{m-1}} \limsup\limits_{\lambda \to 0^+ , x' \to \bar{x} }  G_\lambda(x',v) d\mu_{\zeta}(v).
\end{equation} 
Now, let $v\in F(\bar{x})$ and consider
\begin{equation*}
    w\in \limsup\limits_{\lambda \to 0^+ , x' \to \bar{x} }  G_\lambda(x',v).
\end{equation*}
Then there exist $x_k\to \bar{x}$ and $\lambda_k\to 0^+$ such that $G_{\lambda_k}(x_k,v)\to w$. By Proposition \ref{continuity_rho}, \eqref{def:GradientGe} and since
\begin{equation*}
    \limsup\limits_{\lambda \to 0^+ , x' \to \bar{x} }  \nabla\PMoreau{\Phi}{\lambda}{h}(x',\rho_\lambda(x',v)Lv)=\partial S_\Phi^h(\bar{x},\rho(\bar{x},v)Lv)
\end{equation*}
(see, e.g., \cite[Theorem 3.66, p. 373 ]{Attouch1984}), we have that 
\begin{equation*}
    w=-\theta(\rho(\bar{x},v) ,v) \left( \frac{ x^\ast-\nabla h(x)}{ \left\langle  z^\ast,Lv\right\rangle } \right) \text{ for some } (x^\ast,z^\ast)\in \partial S_\Phi^h(\bar{x},\rho(\bar{x},v)Lv).
\end{equation*}
On the other hand, if $v\in I(\bar{x})$ and 
\begin{equation*}
    w\in \limsup\limits_{\lambda \to 0^+ , x' \to \bar{x} }  G_\lambda(x',v), 
\end{equation*}
we can conclude from \eqref{Lemma_UB02} that $w=0$. Therefore, 
\begin{align*}
\limsup\limits_{\lambda \to 0^+ , x' \to \bar{x} }  G_\lambda(x',v) \subset \left\{  \begin{array}{cc}
    \left\{-\theta(\rho(\bar{x},v) ,v) \left( \frac{ x^\ast-\nabla h(\bar{x})}{ \left\langle z^\ast,Lv\right\rangle } \right)\text{ s.t } (x^\ast,z^\ast)\in \partial S_\Phi^h(\bar{x}, \rho(\bar{x},v)Lv)\right\}  & \text{ if } v\in F(\bar{x})\\
\{ 0\}  & \text{ if } v\in I(\bar{x}),   
\end{array} \right.
\end{align*}
and since $\partial S_\Phi^h(\bar{x}, \rho(\bar{x},v)Lv)$ is single valued, we conclude the proof from \eqref{inclusion_limsup_gradients}.
\smartqed
\end{proof}

\section{Consistency in nonsmooth conic chance constrained optimization problems}
\label{sec:consistency}
 
In this section, we study the convergence of the solutions of optimization problems generated by replacing the probability function by our Moreau regularized versions. Formally, for a fixed reliability parameter $p \in [0,1]$, let us consider a convex proper and lsc  function $\psi:\mathcal{H}\to\Rx$ and the optimization problem
\begin{equation}\tag{$P$}\label{Problem1}
   \begin{aligned}
   \min \psi (x)\\
  \text{s.t }\, x\in M(p)
\end{aligned}
\end{equation}
where $M( {p}) :=\{ x\in \mathcal{H}: \varphi(x) \geq p \}$ and $\varphi$ is the probability function defined in \eqref{Proba:funct}. Furthermore, we consider the family of problems
\begin{equation}\tag{$P_\lambda$}\label{Problemlambda}
   \begin{aligned}
   \min \MoreauYosida{\psi}{\lambda} (x)\\
  \text{s.t }\, x\in M_\lambda(p),
\end{aligned}
\end{equation}
where $M_\lambda(p):= \{ x\in \mathcal{H}: \varphi_\lambda(x) \geq p   \}$ for the regularized probability function $\varphi_\lambda$ given in \eqref{Proba:funct:lamb}. In the same spirit, the objective function of problem \eqref{Problem1} is replaced by its Moreau regularization to have that the optimization problems \eqref{Problemlambda} have smooth data. Let us denote by $v(P)$ and $v(P_\lambda)$ the values of the problems \eqref{Problem1} and \eqref{Problemlambda}, respectively.
 
The first result provides the Painlev\'e-Kuratowski  and Mosco convergence of the feasible sets of problem \eqref{Problemlambda}  to the feasible set given in the original optimization problem \eqref{Problem1}.
 
 	\begin{proposition}\label{mosco_convergence}
 	Consider $p\in [0,1]$. Then,
 	\begin{enumerate}[label=\alph*)]
 		\item The sets $M_\lambda(p) $ Painlev\'e-Kuratowski converge to $M(p).$
 		\item The sets $M_\lambda(p) $ Mosco converge to $M(p)$, provided that the function $h$ in \eqref{scalar01} is sequentially weakly continuous.
 	\end{enumerate}
 	\end{proposition}
 \begin{proof}
 	Let us consider a sequence $x_{k} \in M_{\lambda_k}(p)$  with  $\lambda_{k} \to 0$ and $x_k \to x$ ($ x_{k} \rightharpoonup x$, respectively). Then by  Item b) of Theorem \ref{var_convergence}   we have that 
 	\begin{align*}
 		p \leq \limsup\limits_{k \to \infty} \varphi_{\lambda_k} (x_k)  \leq \varphi(x),
 	\end{align*}
 which shows that $x\in M(p)$. Now, by Item a) of Theorem \ref{var_convergence} we have that 
 $M(p) \subset M_{\lambda}(p)$, which particularly implies that  item a) holds ( item b) holds, respectively). 
 \smartqed
 \end{proof}

It is worth mentioning that Mosco convergence is commonly related to convex sets because of the weak convergence needed in the definition. Nevertheless, a probability function cannot be convex (unless it is a constant mapping) because it takes values on $[0,1]$. Furthermore, the sets $M_\lambda(p)$ are not necessarily convex even in finite dimension, as the following example shows.
 \begin{example}
Let $\xi \sim \mathcal{N}(0,1)$ and consider the probability function $\varphi: \R^2\to\R$ given by
\begin{equation*}
   \varphi(x_1,x_2)=\P(g(x_1,x_2,\xi)\leq 0)
\end{equation*}
where $g:\R^2\times\R\to \R$ is the nonconvex function $g(x_1,x_2,z)=z+f(x_1)+\tfrac{1}{2}x_2^2$ and
\begin{equation*}
  f(x_1)=\left\{
   \begin{array}{cl}
    -\tfrac{1}{2}x_1^2 &\text{ if } x_1\leq 0  \\
     0& \text{ otherwise.} 
   \end{array}
   \right.
\end{equation*}
Consider the convex and differentiable function $h(x_1,x_2)=\tfrac{1}{2}x_1^2$ and notice that
\begin{equation*}
  g(x_1,x_2,z)+h(x_1,x_2)=z+\hat{f}(x_1)+\tfrac{1}{2}x_2^2
\end{equation*}
is convex, where
\begin{equation*}
   \hat{f}(x_1)=\left\{
   \begin{array}{cl}
     \tfrac{1}{2}x_1^2 &\text{ if } x_1>0  \\
     0& \text{ otherwise.} 
   \end{array}
   \right.
\end{equation*}
The regularized probability function is
\begin{equation*}
   \varphi_\lambda(x)=\P\left(\xi-\tfrac{1}{2}\lambda+\MoreauYosida{\hat{f}}{\lambda}(x_1)+\tfrac{1}{2(\lambda+1)}x_2^2\leq \tfrac{1}{2}x_1^2\right)
\end{equation*}
where
\begin{equation*}
   \MoreauYosida{\hat{f}}{\lambda}(x_1)=\left\{
   \begin{array}{cl}
     \tfrac{1}{2(\lambda+1)}x_1^2 &\text{ if } x_1>0  \\
     0& \text{ otherwise.} 
   \end{array}
   \right.
\end{equation*}
The upper level sets $M_\lambda(p)$ of $\varphi_\lambda(x)$ with $\lambda=0.5$ are not all convex:
Figure \ref{figure56a} illustrates the graph of $\varphi_\lambda$ for $\lambda=0.5$ and Figure \ref{figure56b} illustrates its contour plot.

\begin{figure}
    \begin{minipage}[t]{0.5\textwidth}
    \includegraphics[width=3.5in]{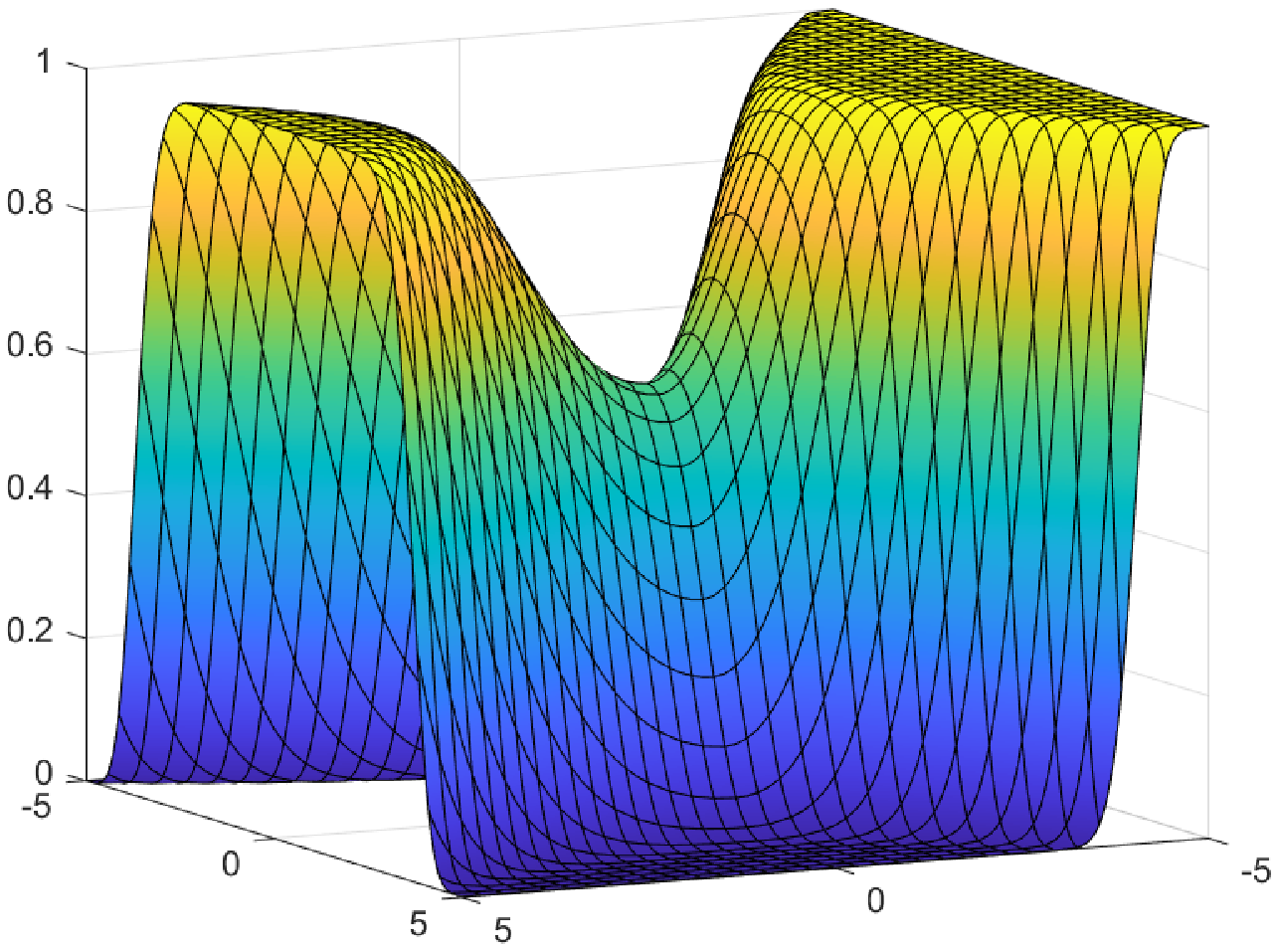}
    \caption[]{Graph of $\varphi_\lambda$ for $\lambda=0.5$}
    \label{figure56a}
    \end{minipage}
    \begin{minipage}[t]{0.5\textwidth}
    \hspace{-0.11cm}
    \includegraphics[width=3.5in]{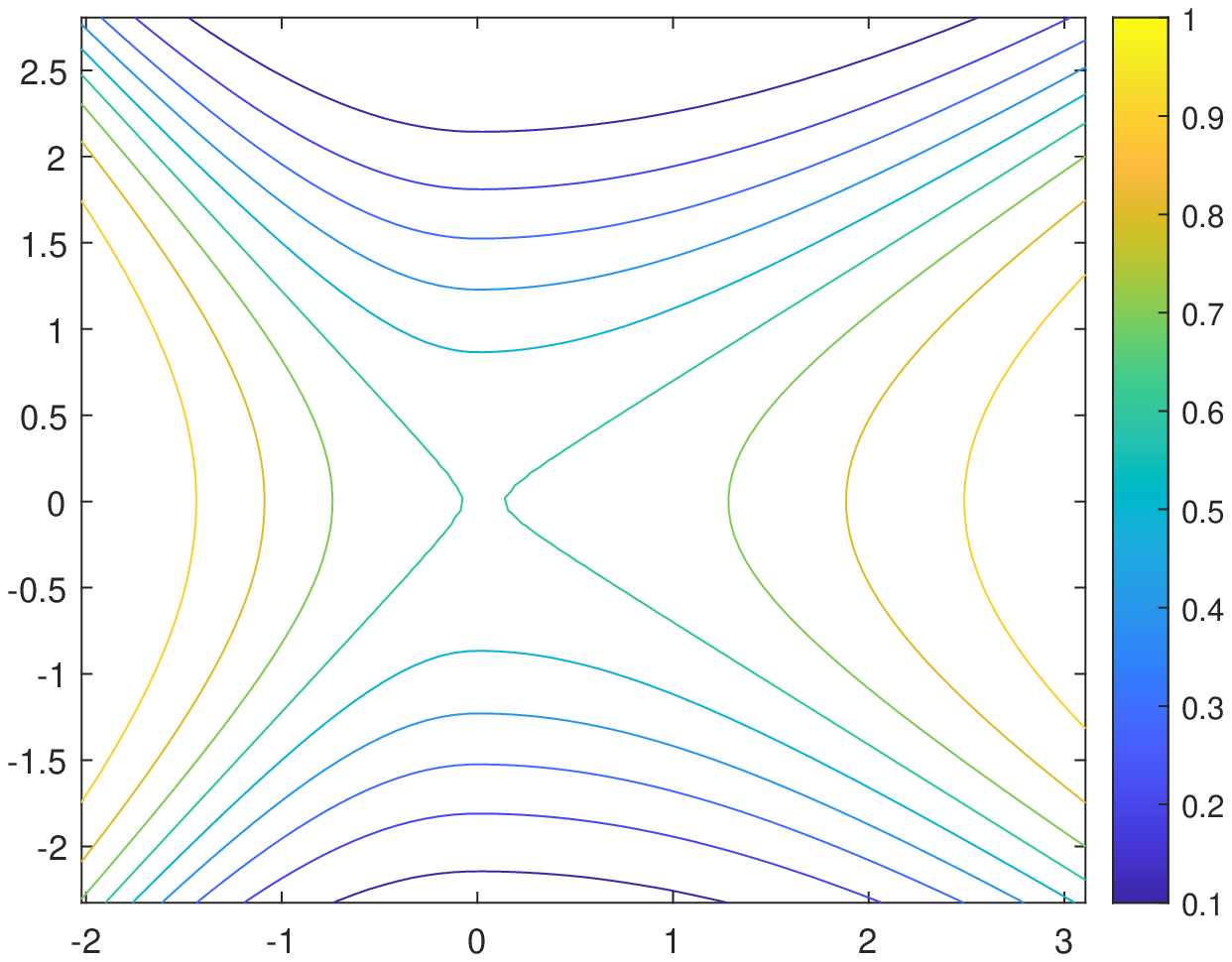}
    \caption[]{contour plot of $\varphi_\lambda$ for $\lambda=0.5$}.
    \label{figure56b}
    \end{minipage}
\end{figure}
\end{example}
 
Here, it is important to notice that our regularized feasible set contains the initial one. Nevertheless, the next proposition shows that under a small enlargement it is possible to show a partial appositive inclusion, which  measures how far we are from our initial feasible set in terms of the random inequality. 
\begin{proposition}
 Let $\mathcal{H}$ be finite dimensional space, and $\epsilon,\eta >0$ be given and $C \subseteq H$ be a bounded closed convex set. Let us define the following enlargements:
 \begin{align*}
M^\epsilon (p) :=\{ x\in \mathcal{H}: \varphi^{ \epsilon}	(x) \geq p	\} \text{ and } \varphi^{ \epsilon}	(x) =\mathbb{P}( S_\Phi(x,\xi) \leq h(x) +\epsilon ).
 \end{align*}
 
Then there exists $\lambda_0>0$ such that for all $p \in \R$ and all $\lambda \in (0,\lambda_0)$ 
  \begin{align}\label{Inclusion01}
	M^\epsilon (p-\eta)\supseteq M_\lambda(p)\cap C,
\end{align} 
In addition, if $f_\xi $ has bounded support we have
\begin{align*}
	M^\epsilon (p)\supseteq M_\lambda(p) \cap C.
\end{align*}
\end{proposition}
 
\begin{proof}
We recall that the probability measure induced by $\xi$ is Borellian and hence tight. Therefore, for any $\eta > 0$, we can find $r>0$ such that $\mathbb{P}( \| \xi \| >r) \leq \eta$. Let us define the set $K = C \times \mathbb{B}_r$.

Then, by Proposition \ref{Uniform:aprox}, we can find $\ell>0$ and $\lambda_1$ such that for all $\lambda \in (0,\lambda_1)$
\begin{align*}
\sup\limits_{(x,z) \in K} | \MoreauYosida{\Phi}{\lambda} (x,z) - S_\Phi(x,z) | \leq \ell \sqrt{\lambda }.
\end{align*}
Now, consider $\lambda_0< \lambda_1$ such that  $\ell \sqrt{\lambda_0 } <\epsilon$. So, for every $x\in C$, the following inclusion is valid:  
\begin{align*}
\{ z \in \mathbb{B}_r : \MoreauYosida{\Phi}{\lambda} (x,z) \leq h(x) \} \subset \{ z \in \mathbb{B}_r : S_\Phi (x,z) \leq h(x) + \epsilon \}.
\end{align*}
As a result, for $x \in M_{\lambda}(p) \cap C$, we have
\begin{align*}
p \leq \varphi_{\lambda}(x) &= \mathbb{P}( \MoreauYosida{\Phi}{\lambda} (x,\xi) \leq h(x) ) = \mathbb{P}(\norm{\xi}\leq r, \MoreauYosida{\Phi}{\lambda} (x,\xi) \leq h(x) ) + \mathbb{P}(\norm{\xi} >  r, \MoreauYosida{\Phi}{\lambda} (x,\xi) \leq h(x) ) \\
&\leq \mathbb{P}(\norm{\xi}\leq r, \MoreauYosida{\Phi}{\lambda} (x,\xi) \leq h(x) ) + \mathbb{P}(\norm{\xi} >  r) \leq \varphi^{\epsilon}(x) + \mathbb{P}(\norm{\xi} >  r).
\end{align*}
From this we can deduce $\varphi^{\epsilon}(x) \geq p - \eta$, i.e., $x \in M^\epsilon (p-\eta)$. If $f_{\xi}$ has bounded support, we may in particular find an appropriate $r$ when $\eta = 0$ is chosen, since then there is $r > 0$ such that 
$\mathbb{P}( \| \xi \| >r) = 0$. Then \eqref{Inclusion01} allows us to conclude.
\smartqed
\end{proof}
 
Now, we provide the main result of this section which establishes a relation between problems \eqref{Problem1} and \eqref{Problemlambda}.

\begin{theorem}\label{MainTheoremOPtPro}
 	Let $\psi:\mathcal{H}\to\Rx$ be a convex, coercive and lower semicontinuous function such that $M(p)\cap\dom \psi\not=\emptyset$. Then
 	
 	\begin{enumerate}[label=\alph*)]
 	   \item $v(P), v(P_\lambda)\in\R$ for all $\lambda>0$ and $v(P_\lambda)\to v(P)$.
 	   \item Let $\lambda_k \to 0^+$ and $(x_{\lambda_k})$ be  any sequence of optimal solutions for $(P_{\lambda_k})$, if $x_{\lambda_k}\rightharpoonup x_0$, then  $x_0$  is an optimum of $(P)$, provided that the function $h$ in \eqref{scalar01} is sequentially weakly continuous. If, furthermore, $\dom\psi=\mathcal{H}$ and $\psi^*$ is Fr\'echet differentiable on $\dom\partial\psi^*$, then $x_{\lambda_k}\rightarrow x_0$.x
 	\end{enumerate}
 \end{theorem}
 
 \begin{proof}
 a) Let $\bar{x}\in M(p)\cap \dom\psi$ be given. By Item d) of Theorem \ref{var_convergence}, the set $M(p)$ is weakly closed, then the nonempty set $M:=M(p)\cap\{x\in\mathcal{H}:\psi(x)\leq \psi(\bar{x})\}$ is weakly compact. Since $\psi$ is weakly lower semicontinuous, by Weierstrass' theorem, $\psi$ has a minimizer in $M$. Therefore, $v(P)\in \R$. Similarly, since the Moreau envelope $\MoreauYosida{\psi}{\lambda}$ is convex, coercive ($0\in \text{int}\dom\psi^*=\text{int}\dom(\MoreauYosida{\psi}{\lambda})^*$), lower semicontinuous and $M(p)\cap \dom\psi\subset M_\lambda(p)\cap \dom\MoreauYosida{\psi}{\lambda}$ for all $\lambda$, we have, similarly, that $v(P_\lambda)\in\R$ for all $\lambda$.
 
To prove
\begin{align*}
   \liminf\limits_{ \lambda \to 0}v(P_\lambda)\geq v(P),
\end{align*} 
let us proceed by contradiction. That is, for some $\alpha<v(P)$, there is a subsequence $x_{\lambda_k}\in M_{\lambda_k}(p)$ with 
\begin{equation}\label{eq:subseqbounded1}
   \MoreauYosida{\psi}{\lambda_k}(x_{\lambda_k})\leq \alpha
\end{equation}
 for all $k$. Since, by Item a) of Proposition \ref{basicprop},
 $\MoreauYosida{\psi}{\lambda_1}(x_{\lambda_k}) \leq \MoreauYosida{\psi}{\lambda_{\lambda_k}}(x_{\lambda_k}) \leq \alpha$ for all $k$ and $\MoreauYosida{\psi}{\lambda_1}$ is coercive, convex and lower semicontinuous, there is a subsequence $x_{\lambda_{k_i}}\rightharpoonup x$. Indeed the level set of $\MoreauYosida{\psi}{\lambda_1}$ is bounded. 
 By \eqref{eq:subseqbounded1} and Item c) of Proposition \ref{basicprop} we get $\psi(x)\leq \alpha$, and on the other hand, by Proposition \ref{mosco_convergence}, $x\in M(p)$. Thus, $v(P)\leq \alpha$, which is a contradiction.

Now, to prove 
\begin{align*}
   \limsup\limits_{ \lambda \to 0}v(P_\lambda)\leq v(P),
 \end{align*}
notice that, by Item a) of Theorem \ref{var_convergence}, $M(p)\subset M_\lambda(p)$ for all $\lambda$ and since $\MoreauYosida{\psi}{\lambda}\leq \psi$ for all $\lambda$ we get $v(P_\lambda)\leq v(P)$ for all $\lambda$.

{b) Let us show that $x_0$ is an optimum of $(P)$. Indeed, by Proposition \ref{mosco_convergence}, we have that $x_0$ is a feasible point of problem $(P)$.  On the one hand, by Proposition \ref{basicprop} item a) and optimality of $x_{\lambda_k}$ we have
 \begin{equation}\label{eq:subseqbounded2}
   \MoreauYosida{\psi}{\lambda_k}(x_{\lambda_k}) \leq \MoreauYosida{\psi}{\lambda_k}(x_0) \leq \psi(x_0)
 \end{equation}
which implies that $\limsup \MoreauYosida{\psi}{\lambda_k}(x_{\lambda_k})  \leq \psi(x_0)$. On the other hand,  by Proposition \ref{basicprop} item c) and optimality of $x_{\lambda_k}$ we have
 \begin{equation}\label{eq:subseqbounded2}
 v(P_{\lambda_k})  = \lim_{k\to \infty}  v(P_{\lambda_k})   =\liminf_{k\to \infty} \MoreauYosida{\psi}{\lambda_k}(x_{\lambda_k}) \geq  \psi(x_0),
 \end{equation} 
 which shows the optimality of $x_0$ and that $\lim_{k\to \infty} \MoreauYosida{\psi}{\lambda_k}(x_{\lambda_k}) =  \psi(x_0)$. Now suppose, furthermore, that $\dom \psi=\mathcal{H}$ and $\psi^*$ is Fr\'echet differentiable on $\dom \partial \psi^*$. It follows that $\dom \partial \psi=\mathcal{H}$. Hence there exists $u\in\partial \psi(x_0)$, which implies $x_0=\nabla \psi^\ast(u)$. Particularly, due to \cite[Theorem 5.2.3]{BV2010}, the function   $\psi(\cdot ) - \langle u,  \cdot\rangle$  attains a strong minimum at $x_0$.  We claim that 
$$\psi(\hat{x}_{\lambda_k}) - \langle u, \hat{x}_{\lambda_k}\rangle\rightarrow \psi(x_0) - \langle u,x_0\rangle,$$
where $\hat{x}_{\lambda_k}=\Prox{\psi}{{\lambda_k}_k}(x_{{\lambda_k}_k})$. Indeed, since $\inf_{z \in \mathcal{H}} \psi(z)>-\infty$ as a result of $\psi$ being convex, coercive and l.s.c., and
\begin{equation*}
   \inf_{z \in \mathcal{H}} \psi(z) + \frac{1}{2{\lambda_k}}\|x_{\lambda_k}-\hat{x}_{\lambda_k}\|^2
   \leq 
   \psi(\hat{x}_{\lambda_k}) + \frac{1}{2{\lambda_k}}\|x_{\lambda_k}-\hat{x}_{\lambda_k}\|^2
   = \MoreauYosida{\psi}{{\lambda_k}}(x_{\lambda_k})\leq \psi(x_0),
\end{equation*}
 we have that
\begin{equation}\label{eq:boundsequenceproximals}
   \|x_{\lambda_k}-\hat{x}_{\lambda_k}\|\leq \sqrt{{\lambda_k}} C,
\end{equation}
for $C \geq \sqrt{(\psi(x_0) - \inf_{z \in \mathcal{H}} \psi(z))} \in \Re$. Thus also $\hat{x}_{\lambda_k}\rightharpoonup x_0$ and
\begin{align*}
   \psi(x_0)\leq & \liminf_{{\lambda_k}\rightarrow 0}\psi(\hat{x}_{\lambda_k})
   \leq \limsup_{{\lambda_k}\rightarrow 0}\psi(\hat{x}_{\lambda_k})\\ \leq& \limsup_{{\lambda_k}\rightarrow 0}\psi(\hat{x}_{\lambda_k})+ \frac{1}{2{\lambda_k}}\|x_{\lambda_k}-\hat{x}_{\lambda_k}\|^2=\limsup_{{\lambda_k}\rightarrow 0}\MoreauYosida{\psi}{{\lambda_k}}(x_{\lambda_k})=\psi(x_0),
\end{align*}
and together yields $\psi(\hat{x}_{\lambda_k}) - \langle u, \hat{x}_{\lambda_k}\rangle\rightarrow \psi(x_0) - \langle u,x_0\rangle$. Therefore, $\|\hat{x}_{\lambda_k}-x_0\|\rightarrow 0$ because $x_0$ is a strong minimum of $\psi(\cdot) - \langle u, \cdot \rangle$, so  by \eqref{eq:boundsequenceproximals}, we can then conclude $\|x_{\lambda_k}-x_0\|\rightarrow 0$.

 }
\smartqed
\end{proof}

Uniqueness of minimizer is intrinsically related with the convexity of the optimization problems. The following result provides conditions under the problems optimization problems \eqref{Problem1} and \eqref{Problemlambda} are convex and consequently all the assumptions of Theorem \ref{MainTheoremOPtPro} hold.

Let us recall that a nonnegative function $f$ defined on a convex set $D\subset \mathcal{H}$ is $\alpha$-concave, where $\alpha\in [-\infty, \infty]$, if for all $x,y\in D$ and all $\lambda\in[0,1]$ the following inequality holds:
 \begin{equation*}
   f(\lambda x+(1-\lambda)y)\geq m_{\alpha}(f(x),f(y),\lambda),
 \end{equation*}
 where $m_\alpha:\R_+\times\R_+\times [0,1]\rightarrow \R$ is defined as follows:
 \begin{equation*}
   m_\alpha(a,b,\lambda)=0\text{ if }ab=0, \;\text{and}\;  \alpha \leq 0
 \end{equation*}
 and for any other value of $a$ and $b$,
 \begin{equation*}
   m_\alpha(a,b,\lambda)=
   \left\{
   \begin{array}{cl}
    a^\lambda b^{1-\lambda} & \text{ if } \alpha=0,  \\
     \max\{a,b\} & \text{ if }\alpha=\infty,\\
     \min\{a,b\} & \text{ if }\alpha=-\infty,\\
    (\lambda a^\alpha+(1-\lambda)b^\alpha)^{1/\alpha} & \text{ otherwise. }
   \end{array}
   \right.
 \end{equation*}
 In the case $\alpha=0$, the function is called log-concave, for $\alpha=1$ concave, and $\alpha=-\infty$ quasi-concave. We also notice that all $\alpha$-concave
functions are quasi-concave. Moreover, the random vector $\xi$ has $\alpha$-concave   probability distribution if   the probability measure $\mathbb{P}_\xi (A) :=\mathbb{P}( \xi \in A) $ induced by $\xi$ on $\mathbb{R}^m$ satisfies that  for any Borel measurable sets $A, B \subseteq \mathbb{R}^m$  and for all $\lambda \in [0,1]$ 
\begin{align*}
    \mathbb{P}_\xi ( \lambda A + (1-\lambda)B) \geq m_\alpha( \mathbb{P}_\xi(A),\mathbb{P}_\xi(B), \lambda).
\end{align*}
 For more details and relations between the (generalized) concavity of random vectors and its density we refer to \cite{Shapiro2014}.
 \begin{corollary}\label{Corollary_mainresult}
	Let us suppose that $\xi$ has an $\alpha$-concave probability distribution and  $\Phi$ satisfies \eqref{scalar01} with $h=0$. Then, for every $\lambda>0$, and any $p \in (0,1)$ the functions $\varphi_\lambda$ and $\varphi$ are $\alpha$-concave on the sets $ \{ x\in \mathcal{H}: \exists z\in \R^m \text{ s.t } \MoreauYosida{\Phi}{\lambda}(x,z)\leq 0  \}$ and $ \{ x\in \mathcal{H}: \exists z\in \R^m \text{ s.t } \Phi(x,z)\leq 0  \}$, respectively. Consequently, for any $p\in (0,1]$ the sets $M_\lambda(p)$ and $M(p)$ are convex. Moreover, suppose that the objective function $\psi$ in the optimization problem \eqref{Problem1} is convex, coercive, lower semicontinuous,  $M(p)\cap\interior(\dom \psi)\not=\emptyset$ and $\psi^*$ is Fr\'echet differentiable on $\dom\partial\psi^*$. Then, the sequence of unique solutions of problems \eqref{Problemlambda} converge to the unique minimizer of \eqref{Problem1}.
 \end{corollary}

\begin{proof}
The $\alpha$-concavity of the functions $\varphi$ and $\varphi_\lambda$ follows from a direct application of 	  \cite[Theorem 4.39, p. 108]{Shapiro2014}. In particular, $\varphi$ and $\varphi_\lambda$ are quasi-concave, hence the sets $M(p)$ and $M_\lambda(p)$, being upper level sets of these functions, are convex. Now suppose that the objective function $\psi$ in the optimization problem \eqref{Problem1} is convex, coercive, lower semicontinuous and $M(p)\cap\interior(\dom \psi)\not=\emptyset$. Then Item a) of Theorem \ref{MainTheoremOPtPro} follows and by \cite[Corollary 16.38]{Combettes2017} we have
\begin{equation*}
  0\in \partial(\psi+\delta_{M(p)})(x_0)=\partial \psi(x_0)+\partial \delta_{M(p)}(x_0),
\end{equation*}
where $\delta_{M(p)}$ is the indicator function of $M(p)$ and $x_0$ is an optimal solution of \eqref{Problem1}. Thus $\partial \psi(x_0)\not=\emptyset$, so the set of optimal solutions of \eqref{Problem1} is a convex subset of $\dom \partial \psi$. By the differentiability assumption over $\psi^*$, the function $\psi$ must be
strictly convex on this set (see, e.g.,  \cite[section 7.3]{BV2010}). Then  \eqref{Problem1} has a unique optimal solution. Similarly, since the objective functions $\MoreauYosida{\psi}{\lambda}$ satisfy the same hypothesis as $\psi$, the problems \eqref{Problemlambda} also have unique optimal solutions. Then, the convergence of optimal solutions  follows from Theorem \ref{MainTheoremOPtPro}.  
\smartqed
\end{proof}
 
\begin{remark}[On extensions]

A  possible extension results when weakening to what is called ``eventual convexity", e.g., \cite{Henrion_Strugarek_2008,vanAckooij_2013,vanAckooij_Malick_2017,vanAckooij_Laguel_Malick_Matiussi-Ramalho_2022}. In this case it could suffice for  $\xi$ follows  an elliptically symmetric distribution. Convexity could only be asserted when $p$ is sufficiently large, but this is usually not a problem in applications. Moreover, it would be necessary to analyse the generalized concavity of the mapping $\rho_{\lambda}$ defined in \eqref{defradialfunction} and likewise for the mapping $\rho$. We leave this for future research. 
\end{remark}

\section{Examples and applications }
\label{sec:examples}

In this section, we review some examples of the potential applications of our results. Formally, we discuss how our approach can be used to rewrite several classes of probability functions arising in (nonsmooth) optimizing models, and consequently, it illustrates the versatility of our research. Our first examples will demonstrate the smoothing effect of the suggested regularization. Then we will examine the situation of a so called ``joint chance constraint". The section will end with the investigation of a situation wherein $\K$ is the cone of positive definite matrices as well as the case wherein $\K$ describes infinitely many inequalities. In each situation we will carefully investigate how \eqref{scalar01} can be concretely shown to hold true.

\subsection{Nonsmooth inequality constraint}
 First, we start our analysis considering a probability function given by a nonsmooth single inequality, that is, 
  \begin{equation*} 
 	\varphi(x) := \mathbb{P}\left(  g (x,\xi ) \leq 0 \right).
 \end{equation*}
 where $g:\mathcal{H}\times \R^m \to \R$ is a (possible nonsmooth) function. It is clear that in that case the cone $\mathcal{K}$  in consideration is given by the set of nonnegative real numbers, the generator of the positive polar cone is nothing more that the singleton   $\mathcal{C}=\{ 1\}$, and our function $\Phi$ is nothing more than the same function $g$. Moreover, in this setting assumption \eqref{scalar01} is equivalent to the existence of a continuously differentiable function $h$ such that $(x,z) \to g(x,z) +h(x) $ is convex. For simplicity, in the following two examples we chose $h(x)=0$ for all $x\in \mathcal{H}$.
 
 \begin{example}
 Let $\xi \sim \mathcal{N}(0,1)$ and consider the nonsmooth function $g:\R\times\mathbb{R}\rightarrow \mathbb{R}$ given by
 \begin{equation*}
   g(x,z)=2f_1(x)+f_2(z)-5,
\end{equation*}
 where $f_1(x)=\max(|x|-1,0)$ and
 \begin{align*}
  	f_2(z) = \left\{  \begin{array}{cl}
  	z^2 & \text{ if } z\geq 0\\
  	  	-z & \text{ otherwise.}
  	\end{array}  \right.
\end{align*}
 The probability function
 \begin{equation*}
   \varphi(x) = \mathbb{P}\left(2f_1(x)-5\leq \xi\leq\sqrt{ -2f_1(x)+5} \right)
 \end{equation*}
 is not differentiable at $\bar{x}=1,-1$. Indeed, the left derivative of $\varphi$ at $\bar{x}=1$ is $\varphi_-'(1)=0$ and the right derivative of $\varphi$ at $\bar{x}=1$ is
 \begin{equation*}
   \varphi_+'(1)=-\tfrac{1}{\sqrt{2\pi}}\left[\tfrac{1}{\sqrt{5}}{\rm exp}(-5/2)+2\rm{exp}(-25/2)\right]<0.
 \end{equation*}
 Similarly, $\varphi$ is not differentiable at $\bar{x}=-1$. Given $\lambda>0$, we have
 \begin{equation*}
   \varphi_\lambda(x) := \mathbb{P}\left(\MoreauYosida{f_2}{\lambda}(\xi) \leq -2\MoreauYosida{f_1}{2\lambda}(x)+5 \right),
 \end{equation*}
where
\begin{align*}
  	\MoreauYosida{f_1}{\lambda}(x) = \left\{  \begin{array}{cl}
  	  f_1(x)	 & \text{ if } |x|\leq 1  \\
  	  |x|-\frac{\lambda}{2}-1 & \text{ if } |x|\geq \lambda+1 \\
  	  \frac{1}{2\lambda}(|x|-1)^2 & \text{ otherwise} 
  	\end{array}  \right.
\end{align*}
and
\begin{align*}
  	\MoreauYosida{f_2}{\lambda}(\xi) = \left\{  \begin{array}{cl}
  	  	-\xi-\frac{\lambda}{2}  & \text{ if } \xi\leq\lambda \\
  	  \frac{1}{2\lambda+1}\xi^2 & \text{ if } \xi\geq 0\\
  	  \frac{1}{2\lambda}\xi^2 & \text{ otherwise.} 
  	\end{array}  \right.
\end{align*}

Figure \ref{figure1a} illustrates the graph of the functions $\varphi_\lambda$ for $\lambda\in \{0, 0.03,0.1,0.3\}$ where $\varphi_0:=\varphi$ and Figure \ref{figure1b} illustrates a zoomed version for $\lambda\in\{0,0.0001, 0.0005\}$ where we can clearly see the smoothness of the regularized probability function $\varphi_\lambda$ at $\bar{x}=-1$.

\begin{figure}
    \begin{minipage}[b]{0.45\textwidth}
    \includegraphics[width=3.2in]{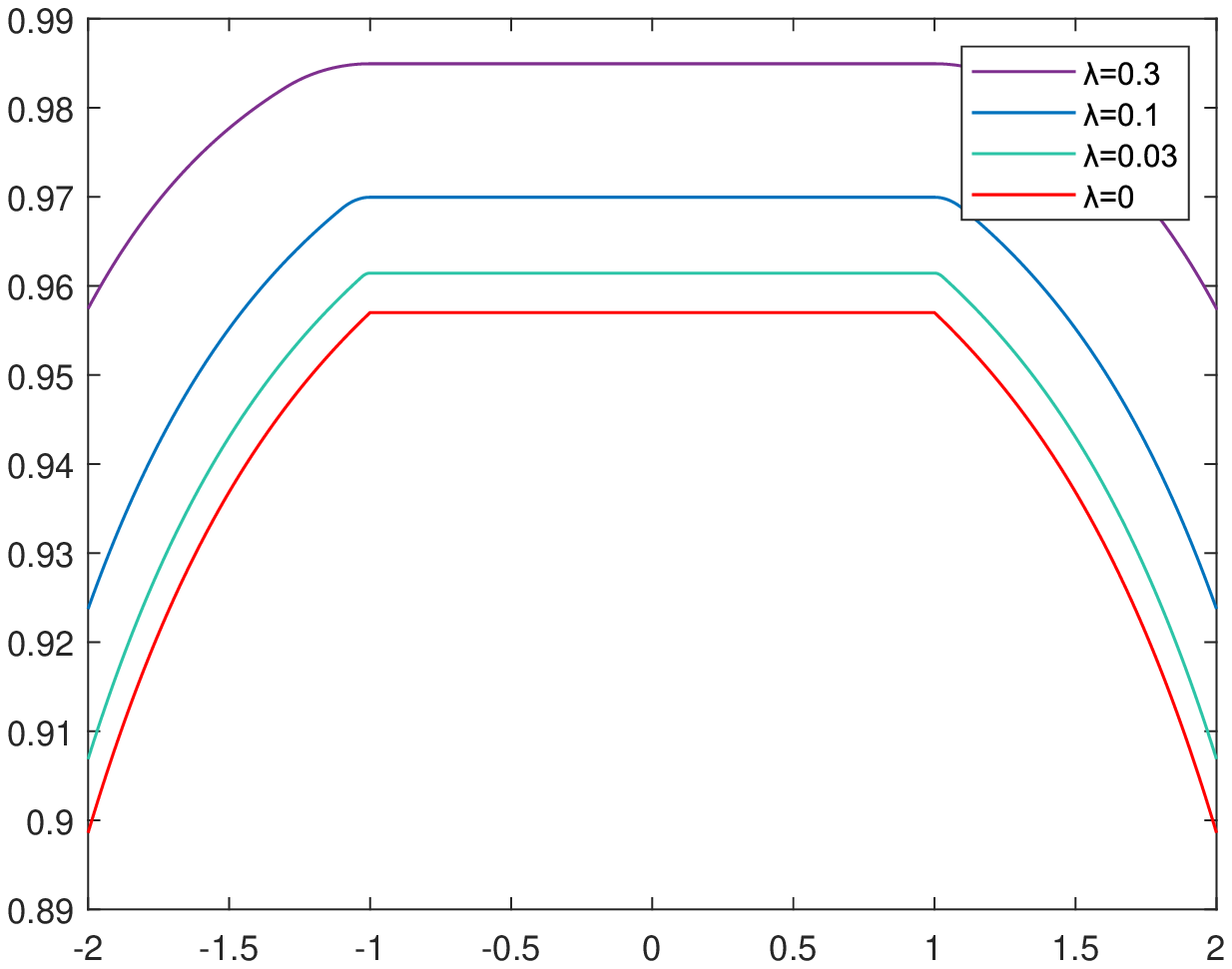}
    \caption{Graph of $\varphi_\lambda$ for $\lambda\in\{0, 0.03,0.1,0.3\}$}
    \label{figure1a}
    \end{minipage}
    \hspace{0.8cm}
    \begin{minipage}[b]{0.45\textwidth}
    \includegraphics[width=3.2in]{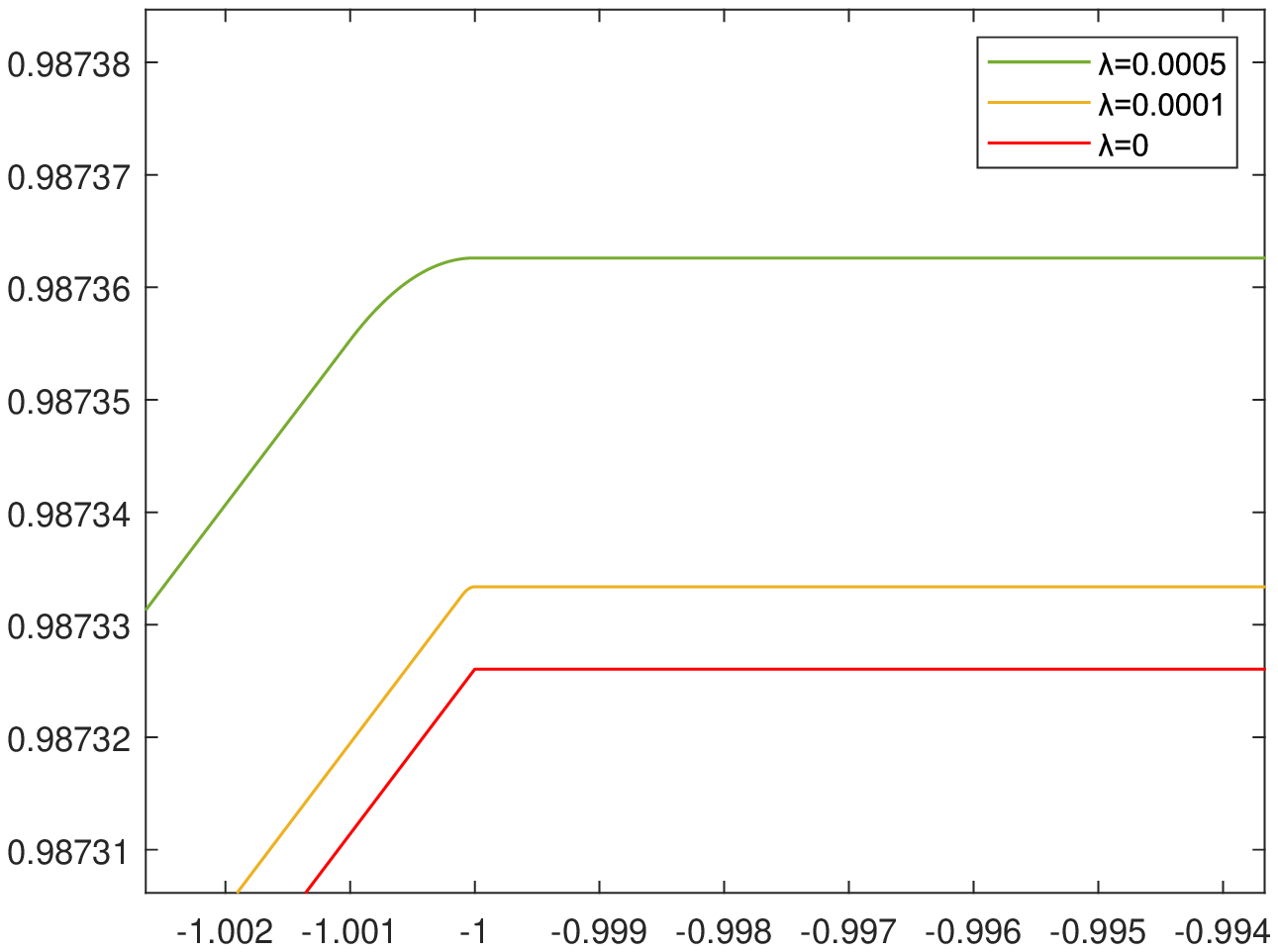}
    \caption{Graph of $\varphi_\lambda$ for  $\lambda\in\{0,0.0001, 0.0005\}$}
    \label{figure1b}
    \end{minipage}
\end{figure}
\end{example}
 
 \begin{example}
 Let $\xi_1,\xi_2 \sim \mathcal{N}(0,1)$ and consider the nonsmooth function $g:\R^2\times\mathbb{R}^2\rightarrow \mathbb{R}$ given by $g(x_1,x_2,z_1,z_2)=f(x_1,x_2)+|z_1|+z_2-3$ where $f(x_1,x_2)=\max(\sqrt{x_1^2+x_2^2}-2,0)$.
 The function
 \begin{equation*}
   \varphi(x_1,x_2) = \mathbb{P}\left(\xi_2\leq -f(x_1,x_2)-|\xi_1|+3\right)
 \end{equation*}
 does not have a directional derivative at $(2,0)$ in the direction $(x_1,x_2)=(1,0)$ since the left derivative is
 \begin{equation*}
   \varphi_-'(2,0):=\lim_{t\to0-}\frac{\varphi(2+t,0)-\varphi(2,0)}{t}=0
 \end{equation*}
and the right derivative is
\begin{equation*}
   \varphi_+'(2,0):=\lim_{t\to0+}\frac{\varphi(2+t,0)-\varphi(2,0)}{t}=-\frac{1}{2\pi}\int_{-\infty}^{\infty}{\rm exp}(-z_1^2+3|z_1|-\tfrac{9}{2})dz_1<0.
 \end{equation*}
 Given $\lambda>0$, we have
 \begin{equation*}
   \varphi_\lambda(x_1,x_2)=\mathbb{P}\left(\xi_2\leq -\MoreauYosida{f}{\lambda}(x_1,x_2)-\MoreauYosida{|\xi_1|}{\lambda}+3\right)
 \end{equation*}
 where
\begin{align*}
  	\MoreauYosida{f}{\lambda}(x_1,x_2) = \left\{  \begin{array}{cl}
  	  f(x_1,x_2)	 & \text{ if } \sqrt{x_1^2+x_2^2}\leq 2  \\
  	  \sqrt{x_1^2+x_2^2}-\frac{\lambda}{2}-2 & \text{ if } \sqrt{x_1^2+x_2^2}\geq \lambda+2 \\
  	  \frac{1}{2\lambda}(\sqrt{x_1^2+x_2^2}-2)^2 & \text{ otherwise} 
  	\end{array}  \right.
\end{align*}
and
\begin{align*}
  	\MoreauYosida{|\xi_1|}{\lambda} = \left\{  \begin{array}{cl}
  	  	\frac{1}{2\lambda}|\xi_1|^2  & \text{ if } |\xi_1|\leq\lambda \\
  	  |\xi_1|-\frac{\lambda}{2} & \text{ otherwise.}
  	\end{array}  \right..
\end{align*}
Figure \ref{figure2} illustrates the nonsmoothness of $\varphi$ on $\{(x_1,x_2):x_1^2+x_2^2=2\}$.

\begin{figure}[b]
\centering
\includegraphics[width=2.5in]{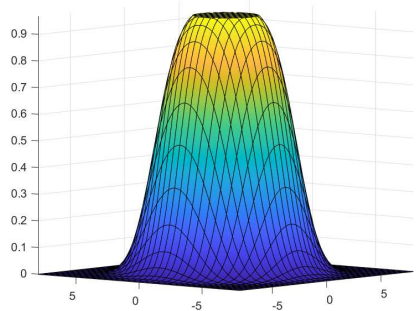}
\caption{Graph of $\varphi$}
\label{figure2}
\end{figure}

 \end{example}
 
\subsection{Joint Chance constraint}
 
Let us consider a family of functions $g_i : \mathcal{H}\times  \mathbb{R}^m \to \R$ with $i=1, \ldots,s$ and the probability function
\begin{equation}\label{Joint}
 	\varphi(x) := \mathbb{P}\left(  g_i(x,\xi ) \leq 0, \text{ for all } i=1, \ldots,s  \right).
\end{equation}
Then, considering $\Phi :\mathcal{H}\times  \mathbb{R}^m \to \R^s$ given by 
\begin{equation}\label{vectorPhi}
 \Phi (x,z):= \begin{pmatrix}
 	g_1(x,z) \\ \vdots\\ g_s(x,z)
 \end{pmatrix}
\end{equation} 
 and the cone $\mathcal{K}:=\R^s_+$, the probability function in \eqref{Joint} can be written as $\varphi(x)=\mathbb{P}( \Phi(x,\xi)\in -\mathcal{K})$, which places us in the framework of \eqref{Proba:funct}. It is easy to see that for a given function $h:\mathcal{H}\to\R$, and considering the unit simplex $\C:=\{x\in\R^s_+:\sum_{i=1}^s x_i=1\}$, effectively ``generating" the positive polar cone of $\K$, we have that $S_\Phi^h(x,z)= \max_{ i=1,\ldots, s} g_i(x,z) + h(x)$. Furthermore, the next proposition gives us a simple characterization of the condition \eqref{scalar01} in terms of the nominal data $g_i$.
\begin{proposition}
 Let $g_i: \mathcal{H}\times \R^m \to \R$ be a family of functions for $i=1,\ldots, s$ and consider the vector valued function $\Phi$ given in \eqref{vectorPhi}. Then the following are equivalent
  \begin{enumerate}[label=\alph*)]
  \item There exists a continuously differentiable convex function $h :\mathcal{H}\to \R$ such that $\Phi$ satisfies \eqref{scalar01}.
\item For every $i=1,\ldots,s$ there exists a continuously differentiable convex function $h_i :\mathcal{H}\to \R$ such that $(x,z) \to g_i(x,z) +h_i(x)$ is convex.
  \end{enumerate}
\end{proposition}

\begin{proof}
To prove a) implies b) consider $w^\ast=e_i$ in \eqref{scalar01} where $e_i$ is the $i$-th standard basic vector of $\R^s$. To prove the converse, let $w^\ast\in \mathcal{C}$ and set $h(x):=\sum_{i=1}^sh_i(x)$. Then since $0\leq w^\ast_i\leq 1$ and the functions $g_i(x,z)+h_i(x)$ and $h_i(x)$ are convex we have that
\begin{equation*}
  	\langle w^\ast,\Phi \rangle (x,z) + h(x) =\sum_{i=1}^s w^\ast_i (g_i(x,z)+h_i(x))+(1-w^\ast_i)h_i(x)
\end{equation*}
is convex.
\smartqed
\end{proof}

 The final example in this subsection illustrates the convergence of the solution and minimizers established in Corollary \ref{Corollary_mainresult}.
 
\begin{example}[Illustrative example]
Let $\xi_1,\xi_2 \sim \mathcal{N}(0,1)$ and consider problem \eqref{Problem1} of section \ref{sec:consistency} with $p=0.95$ and the nonsmooth functions $\psi,\varphi: \R^2\to\R$ given by
\begin{equation*}
\begin{aligned}
  \psi(x_1,x_2)&=|x_1-5|+\tfrac{1}{2}x_2^2+x_2+8\\
  \varphi(x_1,x_2)&= \mathbb{P}( \sqrt{x_1^2+x_2^2}+ |\xi_1 | + \xi_2 \leq 5,\text{ and }  |\xi_1 | + \xi_2 \leq 3)
\end{aligned}
\end{equation*}
In this case, we can consider the vector valued function $\Phi : \R^2 \times \R^2 \to \R^2 $ given by 
\begin{align*}
   \Phi(x_1,x_2,z_1,z_2) = \begin{pmatrix}
   \sqrt{x_1^2+x_2^2}+ |\xi_1 | + \xi_2 -5\\
  |\xi_1 | + \xi_2 - 3
   \end{pmatrix}
\end{align*}
Then, the probability function can be recast as $$\varphi(x_1,x_2)=\mathbb{P}\left( \Phi(x_1,x_2,\xi_1,\xi_2) \in -\mathbb{R}^2_+  \right).$$
In Table \ref{table1} we give the optimal values and the minimizers of \eqref{Problemlambda} associated with problem \eqref{Problem1}. 

\begin{table}[h!!!]
\centering
\begin{tabular}[t]{lcc}
\hline
$\lambda$ & $v(P_\lambda)$ & $x_\lambda$\\
\hline
1& 8.19472 & (2.96739,-1.19475)\\
0.1&10.19347& (2.21702,-0.73521)\\
0.01&10.39840& (2.13857,-0.68601)\\
0.001&10.41892& (2.13071,-0.68105)\\
0.0001&10.42098& (2.12992,-0.68055)\\
0.00001&10.42118& (2.12985,-0.68050)\\
\hline
&$v(P)=$10.42121& $x_0=$(2.12984,-0.68049)\\
\end{tabular}
\caption{\label{table1}Results obtained by MatLab's optimization algorithm {\tt fmincon}.}
\end{table}
\end{example}

\subsection{semidefinite chance constraint}

In this section, we consider the following probability function  
\begin{align}\label{semidefinite}
\varphi(x):=  \mathbb{P}\left( \Phi(x,\xi)  \preceq  0 \right),
\end{align}
where $\Phi: \mathcal{H}\times \mathbb{R}^s \to \mathcal{S}^p$ is a function with $\mathcal{S}^p$   the set of $p\times p$ symmetric matrices, and the symbol $A  \preceq 0$ means that the matrix $A $ is negative semidefinite. It is important to notice that the probability function \eqref{semidefinite} appears as a  natural alternative to deal with semidefinite mathematical programs where there exists a random inflow in the model. 

It is well known that the partial order $\preceq $ can be characterized by the cone of negative definite matrices $\mathcal{S}_{-}^p$. Let us recall that the space $\mathcal{S}^p$ is a Hilbert space endowed with the inner product $\langle A, B \rangle :=\Tr(AB)$, where $\Tr$ represents the trace  operator (see, e.g., \cite{MR1756264}). Using this topological structure, the positive polar cone of $\mathcal{S}_{-}^p$ is given by the set of positive definite symmetric matrices $\mathcal{S}_{+}^p$. It is straightforward to see that the set  $\mathcal{C}:=\{ A \in \mathcal{S}_{+}^p : \Tr(A) =  1 \}$ generates the cone  $\mathcal{S}_{+}^p$. Furthermore, in order to fulfill \eqref{scalar01}, we need to assume an appropriate notion of convexity for this precise setting. The following result establish an equivalent characterization of \eqref{scalar01} through simpler quadratic scalarizations.

\begin{proposition}\label{Proposit}
Let  $\Phi: \mathcal{H}\times \mathbb{R}^m \to \mathcal{S}^p  $ be a function and $h: \mathcal{H}\to  \mathbb{R}$ be a convex and continuously differentiable function. Then, the following are equivalent:
\begin{enumerate}[label=\alph*)]
  		\item For every $A\in \mathcal{C}= \{  A \in \mathcal{S}_{+}^p : \Tr(A) =  1  \}$ the function  $(x,z) \to  \langle A, \Phi(x,z) \rangle + h(x)$ is convex.
  		\item  For every $v\in \mathbb{S}^{p-1}$  the function $(x,z) \to  v^\top \Phi(x,z) v + h(x)$ is convex.
\end{enumerate}
\end{proposition}
 
\proof{Proof.}
On the one hand,  let us suppose that $a)$ holds, and consider a vector $v\in  \mathbb{S}^{p-1}$, that is  ${v}\in \mathbb{R}^p$ with $\| v\|=1$, then let us define the  symmetric  matrix $A:=v v^\top$, which has $\Tr(A)= \| v\|^2=1$. Moreover, the matrix $A$ is positive semidefinite as is clear. 

Finally, $\langle A, \Phi(x,z) \rangle  =  v^\top \Phi(x,z) v$, which shows that the function $x\to v^\top \Phi(x,z) v +h(x)$ is convex and that hence b) holds true.
  	
On the other hand, let us assume that $b)$ holds, and consider $A \in \mathcal{S}^p_+ $ with  $\Tr(A)=1$. Using the   spectral decomposition we have that   the matrix $A$ can be decomposed into  $A= PDP^\top = \sum_{i=1}^p \lambda_i(A) v_i v_i^\top$, where $P$ is a   $p\times p$ orthogonal matrix, and its  columns are the vector $v_i \in\mathbb{R}^p$ with $\| v_i\|=1$, and  $D$ is a  diagonal  given by the eigenvalues  of the matrix $A$, denoted by $ \lambda_1(A),\ldots,  \lambda_p(A)$, allowing for multiplicity. Then, we can compute the inner product of this matrix and $\Phi(x,z)$ by
\begin{align*}
	\langle A, \Phi(x,z) \rangle= \sum\limits_{i=1}^p \lambda_i(A) \langle v_i v_i^\top , \Phi(x,z) \rangle  = \sum\limits_{i=1}^p \lambda_i(A)  v_i^\top  \Phi(x,z)v_i.
\end{align*}
Finally, since $\sum_{i=1}^p \lambda_i(A)  = \Tr(A)=1$ and $\lambda_i(A) \geq 0$, we get that  	$$\langle A, \Phi(x) \rangle + h(x)  = \sum_{ i=1 }^p \lambda_i (A)  \left(  v_i^\top  \Phi(x)v_i+ h(x)    \right) ,$$
 consequently the above function is convex, and that concludes the proof. 
\endproof

\begin{remark}[Matrix convexity]
It is important to mention that Proposition \ref{Proposit} establishes that our desired assumptions hold under the so-called matrix convexity, that is, the assumption that for   every $v\in  \mathbb{S}^{p-1}$  the function $(x,z) \to  v^\top \Phi(x,z) v$ is convex. We refer to \cite[Section 5.3.2]{MR1756264} for more details and references of these properties.
\end{remark}

\begin{example}
Let us consider a family of matrices $A_i, B_j \in \mathcal{S}^p$ for $i=1, \ldots, s$, $j=0,\ldots,m$ and $C^2$ functions $g_i :\mathbb{R}^s \to \mathbb{R}$ for $i=1,\ldots,s$. Define the mapping $\Phi : \mathbb{R}^s \times \mathbb{R}^m \to \mathcal{S}^p $ given by 
$$\Phi(x,z) := \sum_{i=1}^s g_i(x)A_i + \sum_{j=1}^m z_i B_i + B_0.$$

For $i=1,\ldots, s$, consider a convex and continuously differentiable function $h_i$ such that $\pm g_i(x) +h_i(x) $ are convex. 

Let $C>0$ be a  constant greater than any of the absolute values of the eigenvalues of the matrices $A_i$. Then, defining $h := C \sum_{i=1}^s h_i$, we have that for any $v \in \mathbb{S}^{p-1}$, we have that 
\begin{align*}
v^\top \Phi(x,z) v + h(x) =&  \sum_{i=1}^s\left(  v^\top A_i v  g_i(x) + | v^\top A_i v|h_i(x)  ) + (C -| v^\top A_i v| ) h_i(x) \right)\\
& + \sum_{j=1}^m z_i   v^\top B_i v  +  v^\top B_0 v ,
\end{align*} 
is a convex function, which due to Proposition \ref{Proposit} shows that the mapping $\Phi$ satisfies  \eqref{scalar01}.
\end{example}
 
\subsection{Probabilistic/Robust (Probust) Chance Constraint}
Let us consider a compact Hausdorff space $T$ and a  function $g : T \times \mathcal{H}\times  \mathbb{R}^m \to \R$ such that $t\to g(t,x,z)$ is continuous for all  $(x,z) \in \mathcal{H}\times \mathbb{R}^m$. Consider the probability function

\begin{equation}\label{probabilityprobust}
	\varphi(x) := \mathbb{P}\left(  g_t(x,\xi ) \leq 0, \text{ for all } t\in T  \right).
 \end{equation}
Then, let us define $\Phi\colon \mathcal{H}\times  \mathbb{R}^m \to {C}(T)$ given by 
\begin{align}\label{definitionPhiConti}
(x,z) \to \Phi(x,z) \in C(T) \text{ defined by } t\to \Phi (x,z)(t):= g(t,x,z),
\end{align}
where $C(T)$ is the space of continuous functions from $T$ to $\R$ and considering the closed convex cone $\mathcal{K}:=\left\{ f \in C(T) : f(t) \geq 0 \text{ for all }t\in T \right\}$. Using this setting,  we have that the probability function \eqref{probabilityprobust} can be expressed as \eqref{Proba:funct}, that is, $\varphi(x)=\mathbb{P}( \Phi(x,\xi) \in -\mathcal{K})$.

Now, we are going to write the probability function \eqref{probabilityprobust} using a suitable cone $\mathcal{C}$, which generates the positive polar cone of $\mathcal{K}$. In order to do that let us recall some concepts of measure theory. Let us denote by $\mathcal{B}(T)$ the Borel $\sigma$-algebra, which is the smallest $\sigma$-algebra generated by open sets, a signed measure $\mu:  \mathcal{B}(T) \to \R$ is called regular if for every $A\in \mathcal{B}(T)$
\begin{equation*}
\begin{aligned}
   \mu(A)&=\inf\left\{  \mu(U) :U \text{ is open and } A \subset U\right\} = \sup\left\{  \mu(F) : F \text{ is closed and } F \subset A\right\}.
\end{aligned}
\end{equation*}

By Riez' representation theorem (see, e.g., \cite[Theorem 14.14]{MR2378491}) the dual space of $C(T)$ can be identified as the linear space of regular signed measures. Moreover, in this framework the positive polar cone of the set of positive functions is given by the set of (positive) regular measures $\mu:  \mathcal{B}(T) \to \R$ (see, e.g., \cite[Theorem 14.12]{MR2378491}). Consequently a suitable generator of that cone corresponds to the set $\mathcal{C}$ of probability measures on $(T,\mathcal{B}(T))$, which means that our supremum function is given by  
\begin{equation}\label{supfunctionScontinuousfunctions}
  S_\Phi^h(x,z) =\sup \left\{\int_T g(t,x,z)d\mu(t)+h(x):\mu\in \C \right\}.
\end{equation}
The next proposition establishes formally that the general supremum function provided in \eqref{supfunctionScontinuousfunctions} of the vector function \eqref{definitionPhiConti} with scalarization over the set of probability measure is indeed nothing more than the pointwise supremum of the function $g$ with respect to the parameter $t\in T$  plus the function $h$.

\begin{proposition}\label{prop65}
Let $T$ be a compact Hausdorff space and $g : T \times \mathcal{H}\times  \mathbb{R}^m \to \R$ be such that $t\to g(t,x,z)$ is continuous for all  $(x,z) \in \mathcal{H}\times \mathbb{R}^m$. Then, for a given function $h : \mathcal{H}\to \R$ the following holds true:
 \begin{align*}
S_\Phi^h(x,z) = \sup_{t\in T} g(t,x,z) + h(x) \text{ for all } (x,z) \in \mathcal{H}\times \mathbb{R}^m 
 \end{align*}
\end{proposition}

\begin{proof}
Defining  $\mathcal{C}$ of probability measures on $(T,\mathcal{B}(T))$. Fisrt, we have that
\begin{equation}\label{boundSPhi}
  S_\Phi^h(x,z)\leq \sup_{t\in T} g(t,x,z) + h(x)
\end{equation}
for all $(x,z) \in \mathcal{H}\times \mathbb{R}^m$. Moreover, given a point $(x,z) \in \mathcal{H}\times \mathbb{R}^m$, we can take  $\bar{t}\in T$ (since $T$ is a compact Hausdorff space) such that $g(\bar{t},x,z)=\sup_{t\in T} g(t,x,z)$, then if we considering  the  Dirac measure over $\bar{t}$, that is, 
\begin{equation*}
   \mu_{\{\bar{t}\}}(A)=\left\{
   \begin{array}{cl}
    1 & \text{ if } \bar{t} \in A  \\
    0 & \text{ otherwise.}
   \end{array}
   \right.
\end{equation*}
we obtain the equality in \eqref{boundSPhi}, which ends the proof.
\smartqed
\end{proof}

The final result of this section shows a sufficient condition to ensure condition \eqref{scalar01} in the setting of probust chance constrained optimization.
\begin{proposition} \label{prop66}
Let $T$ be a compact Hausdorff space and $g : T \times \mathcal{H}\times  \mathbb{R}^m \to \R$ be such that $t\to g(t,x,z)$ is continuous for all $(x,z) \in \mathcal{H}\times \mathbb{R}^m$. Suppose the existence of a convex continuously differentiable function $h: \mathcal{H}\to \R$, such that for all $t\in T$, the function $ (x,z) \to g(t,x,z) + h(x) $ is convex. Then, the function $\Phi$ defined in \eqref{definitionPhiConti} satisfies condition \eqref{scalar01}.
\end{proposition}

\begin{proof}
For all positive regular measures $\mu:\mathcal{B}(T)\to\R$ we have that \eqref{scalar01} is given by
\begin{equation*}
\langle \mu,\Phi \rangle (x,z) + h(x)=\int_T (g( t,x,z)+h(x))d\mu(t) \text{ for all }(x,z) \in \mathcal{H}\times \mathbb{R}^m,
\end{equation*}
thus its convexity follows from the convexity of the function $g( t,x,z)+h(x)$ for all $(x,z) \in \mathcal{H}\times \mathbb{R}^m$, which is preserved under the integral sign (see, e.g., \cite{Correa2019,MR3947674} and the references therin for more details).
\smartqed
\end{proof}

The final example of this section illustrate an example of a PDE chance constraned optimization problem. 

\begin{example}[PDE chance constrained optimization] 
Let $k\in \mathbb{N}$ and consider the Sobolev space $\mathcal{H}:=W^{2,k}(\mathbb{R}^s)$, where $k>s/2$. By the Sobolev embedding theorem, any $f\in \mathcal{H}$ is H\"{o}lder continuous on $\mathbb{R}^s$. Let $\theta\colon \mathbb{R}^s\to \mathbb{R}$ be a differentiable convex function such that 
$$
\lim_{\Vert y \Vert \to +\infty}\frac{\theta(y)}{\Vert y\Vert}=+\infty.
$$
For $f\in \mathcal{H}$, let us consider the Hamilton-Jacobi equation
\begin{equation}\label{Problem-EDP}
	\begin{aligned}
		\frac{\partial u}{\partial t}(t,x) &+ \theta^{\ast} (\nabla_x u(t,x)) =0 & x\in \mathbb{R}^s, t>0,\\
		u(0,x)&=f(x) + \langle c^\ast, \xi \rangle &x\in \mathbb{R}^s.
	\end{aligned}
\end{equation}
Here $\langle c^\ast, \xi \rangle$ represents a canonical perturbation of a random inflow $\xi$. It is well-known (see, e.g., \cite{MR1736971,MR1757236}) that  problem \eqref{Problem-EDP} admits a unique continuous viscosity solution given by the Hopf-Lax formula
\begin{align}\label{formula_sol}
u_{\xi}(t,x)=\inf\limits_{ y\in  \mathbb{R}^s} \left( f(x-y) + t \theta (y/t)  \right) + \langle c^\ast, \xi \rangle,
\end{align} 
 with the convention $0 \cdot \theta (y/0)= 0 $.
 
Now, given a compact set $ T \subseteq [0,+\infty) \times \mathbb{R}^s$, let us assume that the modeller is expected to minimize a cost functional over the space of functions $\mathcal{H}$ such that the solution \eqref{formula_sol} is greater than or equal to a given level $\alpha$ over the whole set $T$ with high probability level. Formally, it arises to the the following optimization problem
\begin{equation}\label{opt_PDE}
    \begin{aligned}
\min \psi (f) \\ 
\varphi(f) \geq p\\ \quad f\in \mathcal{H}
\end{aligned}
\end{equation}
where $\varphi(f):=\mathbb{P} \left( u_{\xi}(t,x) \geq \alpha ,\text{ for all } (t,x) \in T    \right) $ and $u$ is the unique solution of \eqref{Problem-EDP} for data $f$ given by  the formula \eqref{formula_sol}. 
 
In order to translate the optimization problem \eqref{opt_PDE} into our setting, we set  $g: T\times \mathbb{R}^s \times \mathcal{H}\to\mathbb{R}$, given by    $g(t,x,f,z):=\alpha -  u_z(t,z) $, where  again $u$ is the unique solution of \eqref{Problem-EDP} for data $f$. It is straightforward to show for fixed $(t,x)$ the function $(f,z) \to g(t,x,f,z)$ is convex, and for fixed $(f,z)$ the function $(t,x)\to g(t,x,f,z)$ is continuous. This, by Propositions \ref{prop65} and \ref{prop66} allows us to use the results provided in this paper.
\end{example}

\section{Conclusion}

In this paper we have suggested a regularization \eqref{Proba:funct:lamb} of the probability function given in \eqref{Proba:funct} employing the Moreau envelope. We have shown that this regularization inherits properties of the Moreau envelope itself, namely convergence to the original probability function. Under appropriate, yet mild conditions, convergence can be understood in the Painlevé-Kuratowski or Mosco sense. Furthermore, in a finite dimensional setting, we established continuous differentiability of the regularized probability functions and asymptotic consistency of the resulting gradients. Once again in infinite dimensions, we managed to establish convergence of approximated optimization problems to original problems. 

The above results provide a general approach to handling several nonsmooth chance-constrained optimization problems. We illustrated the possible applications of our inner Moreau regularization covering examples from (nonsmooth) single random inequality to more challenging probabilistic robust optimization models arising from PDE chance-constrained problems. 

Furthermore, the abstract initial conic formulation allows representing general inequality systems inside the probability function, for example, semidefinite constraints. 

It is expected that our convergence results established in Section \ref{sec:consistency} provide the first steps in the development of general algorithms for solving probabilistic constraint programming problems. Besides, the available gradient formula given in Theorem \ref{Theo:form:gradient} provides a suitable representation of the gradient to implement (nonlinear) first decent methods, which be explored in a future research project.

\begin{acknowledgements}
P. P\'erez-Aros was partially supported by ANID-Chile grant: Fondecyt Regular 1200283 and Fondecyt Regular 1190110. C. Soto was supported by the National Agency for Research and Development (ANID)/Scholarship Program/Doctorado Nacional Chile/2017-21170428. E. Vilches was supported by Centro de Modelamiento Matem\'atico
(CMM), ACE210010 and FB210005, BASAL
funds for center of excellence and Fondecyt Regular 1200283 from ANID-Chile. 
\end{acknowledgements}

\bibliographystyle{plain}
\bibliography{references}

 \end{document}